\documentclass{article}
\usepackage[T1]{fontenc}      	
\usepackage[utf8]{inputenc} 	
\usepackage{hyperref} 		 	
\usepackage{graphicx}       	
\graphicspath{{Images/}}

\usepackage{amsmath}      		
\usepackage{amssymb}      		
\usepackage{amsfonts}
\usepackage{amsthm}
\usepackage{dsfont}			 	
\usepackage{caption}         	
\usepackage{relsize}           
\usepackage{nccmath}       		
\usepackage{comment}      		
\usepackage{authblk}			

\usepackage[a4paper]{geometry}

\usepackage{enumitem}        	

\usepackage{xcolor} 			
\definecolor{blue}{RGB}{51,131,255}
\usepackage{bm} 		         
\usepackage{tikz}
\usetikzlibrary{decorations.pathreplacing}

\numberwithin{equation}{section}
\theoremstyle{plain}
	\newtheorem{Theo}{Theorem} 
	\newtheorem{Lemm}{Lemma}[section]
	
	\newtheorem{Prop}{Proposition}[section]
    \newtheorem{Variant}{Variant}[section]
        
\theoremstyle{definition}

	\newtheorem{Hypo}{Hypothesis}

\theoremstyle{remark}
	\newtheorem{Rema}{Remark}[section]

\renewcommand\emptyset{\varnothing}

\newcommand\N{{\mathbb N}}   
    
\newcommand\R{{\mathbb R}}

\newcommand\dd{\,{\textrm d}}
\renewcommand{\div}{\operatorname{div\,}}
\newcommand{\curl}{\operatorname{curl}}

\begin{document}
\title{Exact controllability of the vortex system \\ by means of a single vortex}
\date{\today}
\author[1]{Justine Dorsz \thanks{now at CIRED, AgroParisTechn Cirad, CNRS, EHESS, Ecole des Ponts Paristech, Université Paris-Saclay, Nogent-sur-Marne}}
\author[1]{Olivier Glass}
 
\affil[1]{CEREMADE, Université Paris Dauphine-PSL, Paris}

\maketitle
\begin{abstract}
In this paper, we investigate the controllability of the point vortex system by means of a single vortex. The point vortex system is a well-known simplified model for the incompressible Euler equation, where the vorticity is concentrated in a finite number of Dirac masses. We use one of the vortices as a control, and prove that by suitably choosing its trajectory, we can drive all other vortices to a given prescribed position in arbitrary time.
\end{abstract}
\tableofcontents
%
%
%
%
\section{Introduction}
\subsection{Presentation of the system}
In this paper, we investigate the point vortex system from the viewpoint of control theory.
The vortex system is a classical system in fluid dynamics, whose study originates back to as early as the XIXth century, in particular with works of Helmholtz, Kirchhoff, Kelvin and Poincar\'e. It is a system of ordinary differential equations that can been seen as a simplification of the 2D incompressible Euler (partial differential) equations:
\begin{align}
\partial_t u + (u\cdot\nabla)u + \nabla p & = 0, \label{e.continuity} \\
\text{div}\, u & = 0, \label{e.incomp}
\end{align}
where $u:\R^2 \rightarrow \R^2$ is the velocity field and $p:\R^2 \rightarrow \R$ is the pressure field, as the (scalar) vorticity
\begin{equation} \label{Vorticite}
\omega := \curl u,
\end{equation}
is concentrated on a finite number of points. Classical textbooks on the vortex system are for instance \cite{Lamb,MP,Newton}.
We recall that above, \eqref{e.continuity} represents the conservation of momentum while \eqref{e.incomp} represents the incompressibility constraint, and that the vorticity satisfies the transport equation
\begin{equation} \label{Eq:Vorticite}
\partial_{t} \omega + \div(u \omega) =0.
\end{equation}
To introduce the system, we recall that the velocity at a point $x$ in $\R^2$ generated by the vorticity field $\omega$ is given by the Biot-Savart law:
\begin{equation}\label{e.velo}
u(x) = \frac{1}{2\pi} \int_{\R^2} \frac{(x-y)^{\perp}}{|x-y|^2} \omega(y) \dd y,
\end{equation}
where $x^{\perp}$ stands for the image of $x$ by the rotation of angle $\displaystyle \frac{\pi}{2}$ in $\R^2$,
\begin{equation*}
 x^{\perp} = \begin{pmatrix} x^1 \\ x^2 \end{pmatrix} ^{\perp} = \begin{pmatrix} -x^2 \\ x^1 \end{pmatrix},
\end{equation*}
with $x^i$ the $i^{\textrm{th}}$ component of the vector $x$. 

Correspondingly, when the initial vorticity is concentrated on a finite number $N$ of distinct points of $\R^2$, $x_{1},\ldots,x_{N}$, called vortices, say
\begin{equation*}
\omega = \sum_{i=1}^N \gamma_i \delta_{x_{i}},
\end{equation*}
with $\delta_{x}$ the Dirac measure at the point $x$, and $\gamma_i \neq 0$ the intensity of the vortex located at $x_{i}$, we obtain the corresponding velocity field
\begin{equation*} 
u(x) = \frac{1}{2\pi} \sum_{i=1}^N \gamma_i \frac{(x-x_{i})^{\perp}}{|x-x_{i}|^2}.
\end{equation*}
The vortex system is obtained by neglecting self-interaction, that is, by assuming that each vortex moves under the influence of the velocity field generated by the other vortices only. Then \eqref{Eq:Vorticite} leads to the following system of ordinary differential equations: for $i$ in $\{1,\ldots,N\}$, the position of the $i$th vortex at time $t$ is given by
\begin{equation}\label{e.system}
\left\{ \begin{array}{rcl}
\displaystyle \frac{dx_i}{dt}(t)
&=& \displaystyle \sum_{j=1,\,j\ne i}^{N} \frac{\gamma_j}{2\pi} \frac{(x_i(t)-x_j(t))^{\perp}}{|x_i(t)-x_j(t)|^{2}}, \\
x_i(0) &=& x_{0,i},
\end{array} \right.
\end{equation}
where we have denoted by $x_{0,1}$, \dots, $x_{0,N}$ the initial positions of the vortices. \par
This system may be derived rigorously as the limit of the Euler equations \eqref{e.continuity} and \eqref{e.incomp} when the vorticity of the fluids concentrates on a finite number of points, see for instance \cite{MarchioroPulvirenti:1983,MarchioroPulvirenti:1993,Serfati}.
More recently, it has been proven that this system arises as the limit of an evolution system of small solids in an incompressible perfect fluid, see for instance \cite{GLS:2016} and \cite{GMS:2018}. \par

Let us now introduce the control problem that we investigate on this system. We wish to establish controllability properties of System \eqref{e.system} by means of one of the vortices, or equivalently when we add a vortex of intensity $\gamma^c \neq 0$, whose trajectory we completely control. The dynamics thus reads as follows, for $i$ in $\{1,\ldots,N\}$,

\begin{equation} \label{e.intromain}
\left\{ \begin{array}{rcl}
\displaystyle \frac{dx_i}{dt}(t) 
&=&  \displaystyle \sum_{j=1,\,j\ne i}^{N} \frac{\gamma_j}{2\pi} \frac{(x_i(t)-x_j(t))^{\perp}}{|x_i(t)-x_j(t)|^{2}} 
+ \frac{\gamma^c}{2\pi} \frac{(x_i(t)-y(t))^{\perp}}{|x_i(t)-y(t)|^{2}}, \\
x_i(0) &=& x_{0,i},
\end{array} \right.
\end{equation}
where $y(t)$ is taken as control parameter. This is somewhat reminiscent of the use of some coordinates of a system as a control (see e.g. \cite{Bressan:Impulsive} and references therein), however the rather simple structure of System~\eqref{e.intromain} do not bring us to follow this approach. \par
The question that we raise in this paper is the one of global controllability of System~\eqref{e.intromain}. More precisely, given an arbitrary time $T>0$, initial positions for the vortices $x_{10}$, \dots, $x_{N0}$ and final ones  $x_{1f}$, \dots, $x_{Nf}$, we investigate the possibility of choosing a relevant trajectory $y$ for the controlled vortex, such that the corresponding solution of \eqref{e.intromain} departing from $x_{10}$, \dots, $x_{N0}$ is globally defined in $[0,T]$ and reaches $x_{1f}$, \dots, $x_{Nf}$ at time $T$. \par
The motivation comes from control theory for fluid mechanics that has drawn a large literature in the last thirty years, in particular since the celebrated conference by J.-L. Lions \cite{Lions:AreThere}. Recent progress in the field can be found for instance in \cite{CoronMarbachSueur} where one can find many references concerning works in the subject. In most results in the field, the control takes the form either of an {\it interior} control or a {\it boundary} control. This leads in general to an infinite-dimensional control space (though there are also results using a finite-dimensional, not space-localized control space, such as in the recent work \cite{Nersesyan} --- see also the references therein). Here, by simplifying the model, we may use a much simpler and low-dimensional control to achieve our goal. \par
The motivation for using a single vortex as a control also comes from the theory of fluid-solid interactions. Vortex models can indeed also be obtained as the limit of the evolution of solids in a perfect fluid as the radius of these solids shrink to zero (see \cite{GLS:2016,GMS:2018}). This opens a perspective for the problem of controlling solids inside a fluid {\it by means of another solid}. This seems a quite natural way to apply a control to a fluid system. Recent results on the control of solids immersed in a perfect fluid by means of a boundary control are given in \cite{GKS:2020,GKS:2021}, see also \cite{BoulakiaGuerrero,ImanuvilovTakahashi} for viscous Newtonnian fluids. \par
We finally mention that vortex control was also considered in \cite{VainchteinMezic} (see also references therein): here the perspective is a bit different, since the control is an external, but small, field. \par
%
%
%
%
%
%

\subsection{Main result}
Our main result is as follows.
\begin{Theo} \label{Thm:Main}
System \eqref{e.ode1control} is exactly controllable in arbitrary time, that is to say: given $T>0$, given two $(N+1)$-tuples of distinct points in $\R^{2}$, say
$(x_{0,1}, \dots, x_{0,N},y_{0})$ and $(x_{f,1}, \dots, x_{f,N},y_{f})$, there exists $y \in C^\infty([0,T];\R^2)$ satisfying $y(0)=y_{0}$, $y(T)=y_{f}$, and such that the corresponding solution of \eqref{e.intromain} is defined in $[0,T]$ and satisfies
\begin{equation} \label{target}
(x_{1}(T), \dots, x_{N}(T)) = (x_{f,1}, \dots, x_{f,N}) .
\end{equation}
\end{Theo}
As is classical, a solution of the vortex system is defined as long as vortices do not meet. It is hence a part of the statement of Theorem~\ref{Thm:Main} that the corresponding trajectories of the various vortices (including the one located at $y(t)$) do not cross. Recall that the vortex system can naturally blow up in finite time when the vortices do not have all the same sign (see e.g. \cite{MP}).
%
%
%
%
%
%

\subsection{Notations}
In order to simplify the notations, we omit the factor $1/2\pi$ of the evolution equation describing the dynamics of the system throughout this work, incorporating it in the intensity coefficients $\gamma_{i}$, $\gamma_{c}$. \par
For $x$ in $\R^2$ and for $\bm{x}$ in $(\R^2)^N$, $|x|$ and $\|\bm{x}\|$ denote respectively their Euclidean norm in $\R^2$ and in $(\R^2)^N$. \par
We introduce the following notations for sets of positions of the vortices. 
For $\bm{R} = (R_1,\ldots, R_N)$ in $(\R_+^*)^N$, we denote $D_{\bm{R}}(\bm{x})$ 
\begin{equation} \label{n.ball}
D_{\bm{R}}(\bm{x}):= \left[ B_{R_1}(x_1)\setminus \{x_1\} \right] \times \cdots \times \left[ B_{R_N}(x_N) \setminus \{ x_N\} \right].
\end{equation}
We extend this notation to $D_{\bm{\infty}}(\bm{x}) := \left( \R^2\setminus \{x_1\} \right) \times \cdots 
\times \left( \R^2\setminus \{x_N\} \right)$.
We also introduce 
\begin{equation} \label{not:va}
Y_{\bm{a}} := \left( \R^2\setminus \overline{B_{a_1}(0)} \right) \times \dots \times \left( \R^2\setminus \overline{B_{a_N}(0)} \right).
\end{equation}
Finally, we will denote by $K$ the Biot-Savart kernel:
\begin{equation} \label{Biot-Savart}
K: \left\{ \begin{array}{rcl}
  \R^2\setminus\{0\} & \longrightarrow & \R^2 \setminus \{0 \} \\
   x & \longmapsto & \displaystyle \frac{x^{\perp}}{|x|^2}.
\end{array} \right.
\end{equation}
%
%
%
%
%
%
\subsection{Organization of the paper}
The rest of the paper is devoted to the proof of Theorem~\ref{Thm:Main}. \par
This proof is split in two parts. In Section~\ref{s.onevortex}, we consider a simpler auxiliary problem, when the $N$ vortices are controlled by means of $N$ others.
To do so, we first consider the case $N=1$ in Subsection~\ref{Sub.SimplifiedCase}, and then the general one in Subsection~\ref{s.Nvortex}.
Hence we prove the exact controllability of the vortex system by means of $N$ control vortices, and, under more restrictive assumptions, the possibility to moreover localize the vortices and their control.

Next, Section~\ref{s.singlecontrol} treats the main result regarding the controllability of System \eqref{e.intromain} by means of a single control vortex.
Relying on the ideas of Filippov's convex integration, see \cite{Filippov:1967},
we establish that a single oscillating control playing the role of the $N$ vortices of Section~\ref{s.onevortex} suffices to get a result of exact controllability.
%
%
%
%
\section{Controllability of \texorpdfstring{$N$}{N} vortices by \texorpdfstring{$N$}{N} others}
\label{s.onevortex}
In this section, we consider a simpler problem by putting $N$ controls in the system instead of a single "control vortex".
With as many controls as vortices to control, the intuitive strategy is then to influence each vortex with one dedicated control.
We therefore consider the following control system: for $i$ in $\{1,\ldots,N\}$,
\begin{equation} \label{e.odeN}
\left\{
\begin{array}{rcl}
\displaystyle  \frac{dx_i}{dt}(t) &=&  \displaystyle \sum_{j=1,\,j\ne i}^{N} \gamma_j \frac{(x_i(t)-x_j(t))^{\perp}}{|x_i(t)-x_j(t)|^{2}}
+ \sum_{j=1}^{N} \gamma_j^c \frac{(x_i(t)-y_j(t))^{\perp}}{|x_i(t)-y_j(t)|^{2}},
\, \\
x_i(0) & = & x_{0,i},
\end{array}
\right.
\end{equation}
where the control is given by the $N$ trajectories $y_{1},\dots,y_{N}$. 
Above, $\gamma_1^c,\ldots,\gamma_N^c$, with $\gamma_i^c \neq 0$ for $i \in \{1,\ldots,N\}$, stand for the intensities of the vortices located at $y_1,\ldots, y_N$. 

Precisely, the goal of this section is to prove the following statement concerning the global controllability of $N$ vortices by means of $N$ others.
\begin{Theo} \label{t.NvortexNcontrols}
For all $\bm{x_{0}}= (x_{1,0}, \dots, x_{N,0})$ and $\bm{x_f}= (x_{f,1},\dots, x_{f,N}) $ in $(\R^{2})^N$, with $x_{0,i}\ne x_{0,j}$ and $x_{f,i}\ne x_{f,j}$ for $i,j$ in $\{1,\ldots,N\}$ and $i\ne j$, and for all $T>0$, it is possible to find a control $\bm{y} = (y_1,\ldots,y_N) \in C^{\infty}([0,T],\R^2)^{N}$ such that the corresponding solution of \eqref{e.odeN} is well-defined in $[0,T]$ and satisfies $\bm{x}(T)=\bm{x_f}$. \par
In particular the $2N$ points $x_{i}$, $i=1\ldots N$, and $y_{j}$, $j=1\ldots N$, avoid each other: for all $i, j$ in $\{1,\ldots,N\}$, for all $t$ in $[0,T]$, $x_i(t) \neq y_i(t)$, and for $j\neq i$ $x_i(t) \neq x_j(t)$ and $y_i(t) \neq y_i(t)$.
\end{Theo} \par
In this section, we prove Theorem~\ref{t.NvortexNcontrols}, and give at the end a variant of it to localize the trajectories under restrictive assumptions.
The general idea of the proof is to place each control vortex $y_i$ very close to the vortex $x_i$ in order to ensure the control of the trajectory of $x_i$,
by relying on the fact the interaction between two close vortices leads to a dominant term in the dynamics of System~\eqref{e.odeN}.
The system in the case $N=1$ concerns only the interaction between a vortex and its control, which gives the dominant term in the case $N \geq 2$.
Therefore, we begin with the elementary case of a single vortex.
\subsection{A simple case: a single vortex controlled by another}
\label{Sub.SimplifiedCase}
We first establish preliminary results regarding the controllability of a single vortex under the action of one control, that is, we consider the case $N=1$. 
Hence we are reduced to the following system describing the evolution of the position of the vortex $x$, under the action of the control vortex $y$:
\begin{equation} \label{e.evolution1vortex}
\left\{
\begin{array}{rcl}
\displaystyle \frac{dx}{dt}(t) & =& \displaystyle  \gamma \frac{(x(t)-y(t))^{\perp}}{|x(t)-y(t)|^{2}}, \\
x(0) & =& x_0,
\end{array}
\right.
\end{equation}
with  $\gamma \neq 0$ and $x_0$ in $\R^2$. \par
In that case one can prove the following statement.
\begin{Prop} \label{t.onevortexdynamic}
For all $x_0$ and $x_f$ in $\R^2$, $T>0$ and $y_{0} \in \R^2 \setminus \{ x_{0} \}$, it is possible to find a control $y\in C^{\infty}([0,T],\R^2)$ satisfying $y(0)=y_{0}$, such that $x$ the corresponding solution of~\eqref{e.evolution1vortex} is well-defined (in particular $\displaystyle \min_{t \in [0,T]} |x(t) - y(t)| >0$) and satisfies $x(T)=x_f$.
\end{Prop}
Before going into the proof of Proposition~\ref{t.onevortexdynamic}, we state an elementary lemma regarding the right-hand side of System~\eqref{e.evolution1vortex}.
\begin{Lemm} \label{t.onevortexstationnary}
Let $x$ in $\R^{2}$ and $\gamma$ in $\R^*$. The map $f_x: \R^2 \setminus \{ x \}  \longrightarrow \R^2 \setminus \{ 0\}$,
$y \longmapsto \gamma \frac{(x-y)^{\perp}}{|x-y|^2}$, is a $C^\infty$ diffeomorphism whose inverse is given by $f_x^{-1}:v \longmapsto x + \gamma \frac{v^{\perp}}{|v|^2}$.
\end{Lemm}
The proof is straightforward. 
We can now prove Proposition~\ref{t.onevortexdynamic}.
\begin{proof}[Proof of Proposition~\ref{t.onevortexdynamic}]
We first introduce a $C^{\infty}$ curve $\Gamma:[0,T] \rightarrow \R^2$ satisfying the conditions
\begin{equation}\label{Eq:CondGamma}
\begin{gathered} 
\Gamma(0) = x_{0} \ \text{ and } \ \Gamma(T)=x_{f}, \\
\forall t \in [0,T], \ \dot{\Gamma}(t) \neq 0, \\
\dot{\Gamma}(0) = \gamma \frac{(x_{0}-y_{0})^{\perp}}{|x_{0}-y_{0}|^{2}}.
\end{gathered}
\end{equation}
It is elementary to construct such a curve. Now we define the control according to Lemma~\ref{t.onevortexstationnary}:
\begin{equation} \label{e.yexplicit}
y(t):= f^{-1}_{\Gamma(t)} (\dot{\Gamma}(t)) =\Gamma(t) + \gamma \frac{\dot{\Gamma}(t)^\perp}{|\dot{\Gamma}(t)|^2} \ \text{ for } \ t \in [0,T]. 
\end{equation}
The previous conditions ensure that $y \in C^{\infty}(\left[0,T\right],\R^2)$ and $y(0)=y_{0}$.\\
Now replacing the control in \eqref{e.evolution1vortex} by the expression \eqref{e.yexplicit}, we see that $\Gamma$ satisfies System~\eqref{e.evolution1vortex}, and hence coincides with its unique solution $x$. This ensures in particular that $x(T) = x_f$. 
\end{proof}
\begin{Rema}
We actually prove a stronger statement, since we can actually make $x$ follow any trajectory satisfying the conditions \eqref{Eq:CondGamma}. 
\end{Rema}
We now state a variant of Proposition~\ref{t.onevortexdynamic} that will play an important role in the proof of Theorem~\ref{Thm:Main}. The idea is that, if we can moreover choose the starting point $y_{0}$ of the control $y$, then one can improve the description of the control trajectory $y$ and of the controlled vortex $x$.
\begin{Variant} \label{Var:onevortex}
Let $x_{0} \neq x_f$ in $\R^2$, $T>0$ be given.
Then one can find a control $y\in C^{\infty}([0,T],\R^2)$ such that the corresponding solution $x$ of~\eqref{e.evolution1vortex} and the control $y$ have disjoint straight-lined trajectories
in $\overline{B}(x_{0},|x_{f}-x_{0}|)$ and $\overline{B}(y(0),|x_{f}-x_{0}|)$, respectively. Moreover for $x_f$ sufficiently close to $x_0$ (depending on $T$), the balls $B(x_0, 4|x_f - x_0|)$ and $B(y(0), 4|x_f - x_0|)$ do not intersect.
\end{Variant}
\begin{proof}
We let $x$ follow the straight line from $x_0$ to $x_f$ at constant speed and correspondingly define: 
\begin{equation*}
\Gamma:t \mapsto x_0 + \frac{t}{T}(x_f-x_0).
\end{equation*}
Then as in Lemma~\ref{t.onevortexstationnary} we define $y$ by \eqref{e.yexplicit}, which gives a straight-line trajectory for $y$ with moreover
\begin{equation} \label{EsXY} 
|x(t) - y(t)| = \frac{\gamma T}{|x_{f} - x_{0}|} \ \text{ and } \ | y(t) - y_{0}| \leq |y(T) - y(0)| = |x_f-x_0|. 
\end{equation}
%
Moreover, if we have
\begin{equation} \label{x0xfproches}
|x_f - x_0|^2 \leq \frac{|\gamma|T}{8},
\end{equation} 
then $|x_0 - y(0)| = \frac{|\gamma| T}{|x_f - x_0|} \geq 8|x_f - x_0|$ and the disjunction of the balls $B(x_0, 4|x_f - x_0|)$ and $B(y(0), 4|x_f - x_0|)$ follows.
\end{proof}
\subsection{The general case of \texorpdfstring{$N$}{N} vortices controlled by \texorpdfstring{$N$}{N} other vortices}
\label{s.Nvortex}
In this subsection we extend the results of Subsection \ref{Sub.SimplifiedCase} to $N \geq 2$: we consider the general case of $N$ vortices $\bm{x} = (x_1,\ldots, x_N)$, controlled by $N$ others, $\bm{y} = (y_1,\ldots, y_N)$. The dynamics of the complete system is given by Equation~\eqref{e.odeN} \par
Let us rewrite System~\eqref{e.odeN} as follows: for $i$ in $\{1,\ldots,N\}$,
\begin{equation}\label{e.odeN2}
\left\{
\begin{array}{rcl}
\displaystyle \frac{dx_i}{dt}(t) &=& \displaystyle F_{\bm{x},i}(\bm{y}) (t),
\, \\
x_i(0) &=& x_{0,i},
\end{array}
\right.
\end{equation}
where the right-hand side
$F_{\bm{x}}= ( F_{\bm{x},1}(\bm{y}),\dots,F_{\bm{x},N}(\bm{y}) )$ 
is given for $i$ in $\{1,\ldots,N\}$ by
\begin{equation} \label{e.Fintro}
F_{\bm{x},i}(\bm{y}) := \sum_{j=1,\,j\neq i}^{N} \gamma_j \frac{(x_i-x_j)^{\perp}}{|x_i-x_j|^{2}} + \sum_{j=1}^{N} \gamma_j^c\frac{(x_i-y_j)^{\perp}}{|x_i-y_j|^{2}},
\end{equation}
for ${\bm x}, {\bm y} \in (\R^{2})^N$ such that $x_{j} \neq x_{k}$ for $j \neq k$ and $x_{j} \neq y_{k}$ for all $j$ and $k$. \par
\ \par
Before proving Theorem~\ref{t.NvortexNcontrols}, we first consider a time-independent problem reminiscent of Lemma~\ref{t.onevortexstationnary}. We first establish a regularity result regarding a simplified version $\tilde{F}_{\bm{x}}$ of $F_{\bm{x}}$, then we show how $\tilde{F}_{\bm{x}}$ approximates the right-hand side of \eqref{e.odeN2} when each control $y_i$ is very close to the vortex $x_i$ for $i$ in $\{1,\dots,N\}$. This allows to establish the surjectivity of the mapping $F_{\bm{x}}$ onto some subset of $(\R^{2})^N$. We recall the notations \eqref{NotBall1}, \eqref{n.ball}, and \eqref{not:va}.
\begin{Lemm} \label{t.yexplicit.dimN}
Let $(x_1,\ldots,x_N)\in (\R^{2})^N$, with $x_i\ne x_j$ for $i,j\in \{1,\ldots,N\}$ and $i\ne j$. Let
\begin{equation*}
\renewcommand{\arraystretch}{1.5}
\begin{array}{rcl}
\tilde{F}_{\bm{x}}:& D_{\infty}(\bm{x})  &\longrightarrow  D_{\infty}(\bm{0})  \\
 & (y_1,\ldots,y_N) & \longmapsto \left( \displaystyle  \gamma_1^c \frac{(x_1-y_1)^{\perp}}{|x_1-y_1|^2},\ldots,\gamma_N^c \frac{(x_N-y_N)^{\perp}}{|x_N-y_N|^2} \right).
\end{array}
\end{equation*}
Then $\tilde{F}_{\bm{x}}$ is a $C^\infty$-diffeomorphism. Moreover for $\bm{r}=(r_1,\cdots,r_N) \in (\R_+^*)^N$, setting $\bm{a}:= \left(\displaystyle |\gamma_1^c|\frac{1}{r_1},\cdots,\displaystyle |\gamma_N^c|\frac{1}{r_N}\right)$, we have 
\begin{equation*}
\tilde{F}_{\bm{x}}\left( D_{\bm{r}}(\bm{x})\right) = Y_{\bm{a}}.
\end{equation*}
\end{Lemm}
\begin{proof}[Proof of Lemma~\ref{t.yexplicit.dimN}]
The mapping $\tilde{F}_{\bm{x}}$ clearly is of class $C^\infty$ and it is elementary to check with Lemma~\ref{t.onevortexstationnary} that it is invertible and that its inverse $\tilde{F}_{\bm{x}}^{-1}$ is given by:
\begin{equation} \label{eq.Ftildeinv}
\begin{array}{rcl}
\tilde{F}_{\bm{x}}^{-1}:&  D_{\infty}(\bm{0})  &\longrightarrow   D_{\infty}(\bm{x}) \\
 & (v_1,\ldots,v_N) & \longmapsto \left(x_1 + \gamma_1^c \displaystyle \frac{v_1^{\perp}}{|v_1|^2},\ldots,x_N + \gamma_N^c \displaystyle \frac{v_N^{\perp}}{|v_N|^2} \right).
\end{array}
\end{equation}
Moreover we have immediately $\tilde{F}_{\bm{x}}\left( D_{\bm{r}}(\bm{x})\right) \subset Y_{\bm{a}}$ and $\tilde{F}^{-1}_{\bm{x}}(Y_{\bm{a}}) \subset D_{\bm{r}}(\bm{x})$ with $\bm{a}$ defined above.
\end{proof}
We can now deduce the following proposition on the mapping $F_{\bm{x}}$ introduced in \eqref{e.Fintro}. \par 
\begin{Prop} \label{t.nvorticesautonomous}
Let $(x_1,\ldots,x_N)\in (\R^{2})^N$, with $x_i\ne x_j$ for $i,j \in \{1,\ldots,N\}$ and $i\ne j$.
There exists $\bm{a}$ in $(\R_+^*)^N$ depending only on $(\gamma_{i})_{i=1..N}$, $(\gamma_{i}^c)_{i=1..N}$ and $\min_{i \neq j} |x_{i} - x_{j}|$,
with $\bm{a}$ decreasing as $\min_{i \neq j} |x_{i} - x_{j}|$ increases, such that for all $\bm{v}$ in  $Y_{\bm{a}}$, there exists $\bm{y} \in \R^{2N}$ such that $F_{\bm{x}}(\bm{y})=\bm{v}$. 
Moreover $F_{\bm{x}}$ realizes a $C^1$ diffeomorphism from some neighborhood $\mathcal{V}_{\bm x}$ of $\bm x$ in $\R^{2N}$ to $Y_{\bm{a}}$, and there exists $K_a>0$ depending on $\bm{a}$ such that for all $\bm{v}$ in $Y_{\bm{a}}$ and $\bm{\tilde{y}}$ in $\mathcal{V}_{\bm x}$, 
\begin{equation} \label{eq.iter}
\left\| \bm{\tilde{y}} - F_{\bm{x}}^{-1}(\bm{v}) \right\| \leq K_a \left\|\tilde{F}_{\bm{x}}^{-1}(\bm{v}) - \tilde{F}^{-1}_{\bm{x}} \circ F_{\bm{x}}(\bm{\tilde{y}}) \right\|.
\end{equation}
\end{Prop} \par
\begin{proof}[Proof of Proposition~\ref{t.nvorticesautonomous}]
This is proven in six steps. To complete Notation \ref{n.ball}, for $\bm{R} = (R_1,\ldots, R_N)$ in $(\R_+^*)^N$, we introduce
\begin{equation} \label{NotBall1}
B_{\bm{R}}(\bm{x}) := B_{R_1}(x_1) \times \cdots \times B_{R_N}(x_N).
\end{equation}

\par	
\noindent
\textbf{1. Restriction of the domain.} 
We first set for $i$ in $\{1,\ldots,N\}$,
\begin{equation} \label{DefR}
R_i := \displaystyle \frac{ \displaystyle \min_{1\leq j \leq N}|\gamma_j^c| \min_{j \neq i} |x_i-x_j| }{8 (N-1) \displaystyle \max_{1\leq j\leq N} (|\gamma_j^c|,|\gamma_j|) }.
\end{equation}
Note in particular that $R_i \leq \displaystyle \frac{1}{8} \min_{1\leq j \leq N} |x_i -x_j|$, thus $B_{R_i}(x_i) \cap B_{R_j}(x_j) = \emptyset$ for $j\ne i$. 
Now for ${\bm y} \in D_{\bm{R}}(\bm{x})$ according to Notation \ref{n.ball}, and for $i,\, j \in \{1,\dots,N\}$, $j\neq i$,
we have
\begin{equation} \label{XiYi}
|x_i-y_i| \leq R_i \ \text{ and } |x_i-y_j| \geq \frac{7|x_i-x_j|}{8}.
\end{equation}
Hence we deduce
\begin{equation*} 
\sum_{j \neq i}\displaystyle \left( \frac{1}{|x_i-x_j|}  + \frac{1}{|x_i-y_j|}   \right) \leq
\frac{15}{7} \sum_{j \neq i}\displaystyle \frac{1}{|x_i-x_j|} ,
\end{equation*}
and
\begin{eqnarray}
\nonumber
\frac{|\gamma_i^c|}{|x_i-y_i|}
&\geq&  8 \displaystyle \max_{1\leq j\leq N} (|\gamma_j^c|,|\gamma_j|) \sum_{j \neq i}\displaystyle  \frac{1}{|x_i-x_j|} \\
\label{Eq2.8bis}
&\geq& \frac{56}{15} \sum_{j\neq i} \left( \frac{|\gamma_j|}{|x_i-x_j|} + \frac{|\gamma_j^c|}{|x_i-y_j|} \right).
\end{eqnarray}
Consequently on $D_{\bm{R}}(\bm{x})$,  for all $i \in \{1,\ldots,N\}$, one has
\begin{equation} \label{Eq:minF}
| F_{\bm{x},i}(\bm{y}) | \geq \frac{|\gamma_i^c|}{|x_i-y_i|} - \sum_{j\neq i} \left( \frac{|\gamma_j|}{|x_i-x_j|} + \frac{|\gamma_j^c|}{|x_i-y_j|} \right) 
> 2 \sum_{j\neq i} \left( \frac{|\gamma_j|}{|x_i-x_j|} + \frac{|\gamma_j^c|}{|x_i-y_j|} \right) .
\end{equation}
\ \par \noindent
\textbf{2. Decomposition of $F_{\bm{x}}$.} 
For $\bm{y}$ in $D_{\bm{R}}(\bm{x})$, we decompose $F_{\bm{x}}(\bm{y})$ as the sum of two contributions:
\begin{equation*}
F_{\bm{x}}(\bm{y}) = \tilde{F}_{\bm{x}}(\bm{y}) + G_{\bm{x}}(\bm{y}),
\end{equation*}
where $\tilde{F}_{\bm{x}}$ was introduced in Lemma \ref{t.yexplicit.dimN}, and, for $i$ in $\{1,\ldots,N\}$,
\begin{equation*}
G_{\bm{x},i}(\bm{y}) := \displaystyle \sum_{j\ne i}^{N} \gamma_j \frac{(x_i-x_j)^{\perp}}{|x_i-x_j|^{2}}
+ \sum_{j\ne i}^{N} \gamma_j^c \frac{(x_i-y_j)^{\perp}}{|x_i-y_j|^{2}}.
\end{equation*}
Note that $G_{\bm{x}}$ is a $C^\infty$ mapping on $D_{\bm{R}}(\bm{x})$, as a rational function without pole in $B_{\bm{R}}(\bm{x})$. Moreover, according to  \eqref{DefR}-\eqref{XiYi}, for all $k \in \N$, we can find a constant $C^G_{k}$, depending only on $k$, $(\gamma_{i})_{i=1..N}$, $(\gamma_{i}^c)_{i=1..N}$ and $\min_{i \neq j} |x_{i} - x_{j}|$ such that for all $\bm{y}$ in $D_{\bm{R}}(\bm{x})$, for all $i \in \{1,\dots,N\}$, we have
\begin{equation} \label{e.Gbound}
\left\| D^k G_{\bm{x},i}(\bm{y}) \right\| \leq C^{G}_{k}.
\end{equation}
Moreover, according to \eqref{Eq2.8bis} we have
\begin{equation} \label{EGsurF}
\left|\frac{G_{\bm{x},i}(\bm{y})}{\tilde{F}_{\bm{x},i}(\bm{y})}\right| < \frac{1}{3} \text{  on  } D_{\bm R}({\bm x}).
\end{equation}
\ \par	
\noindent
\textbf{3. Study of the mapping $\tilde{F}^{-1}_{\bm{x}} \circ F_{\bm{x}}$.} 
We now introduce the mapping 
\begin{equation} \label{DefJ}
J_{\bm{x}} := \tilde{F}_{\bm{x}}^{-1} \circ F_{\bm{x}}.
\end{equation}
that is well-defined of class $C^\infty$ on $D_{\bm{R}}(\bm{x})$ since $F_{\bm{x},i}$ does not vanish for $i=1,\ldots,N$ according to \eqref{Eq:minF}.
Let us study the behavior of $J_{\bm{x},i}$ near $x_i$, for fixed $i$ in $\{1,\dots,N\}$. Starting from
\begin{equation*}
J_{\bm{x},i}(\bm{y}) - x_i = 
 \gamma_i^c \displaystyle \frac{\left( \tilde{F}_{\bm{x},i}(\bm{y})+G_{\bm{x},i}(\bm{y}) \right)^{\perp}}{\left| \tilde{F}_{\bm{x},i}(\bm{y})+G_{\bm{x},i}(\bm{y}) \right|^2},
 \end{equation*}
and relying on
\begin{equation} \label{AbsTildeF}
\frac{1}{\left| \tilde{F}_{\bm{x},i}(\bm{y}) \right|^2} = \frac{|x_{i}-y_{i}|^2}{|\gamma_{i}^c|^2},
\end{equation}
we can write
\begin{equation} \label{ExprJ}
J_{\bm{x},i}(\bm{y}) - x_i  =  \gamma_i^c \displaystyle 
\frac{\left( \tilde{F}_{\bm{x},i}(\bm{y})+G_{\bm{x},i}(\bm{y}) \right)^{\perp}}{\left| \tilde{F}_{\bm{x},i}(\bm{y}) \right|^2} \frac{1}{1 + H_{\bm{x},i}({\bm y})},
\end{equation}
where $H_{\bm{x},i}(\bm{y})$ is a smooth function given on $B_{\bm R}({\bm x})$ by
\begin{equation} \label{DefH}
H_{\bm{x},i}({\bm y}):=
 2 \left\langle \frac{(x_{i}-y_{i})^\perp}{\gamma_{i}^c}, G_{\bm{x},i}(\bm{y}) \right\rangle
+ \frac{|x_{i}-y_{i}|^2}{|\gamma_{i}^c|^2} \left| G_{\bm{x},i}(\bm{y}) \right|^2 \ \text{ on }\  
B_{\bm R}({\bm x}).
\end{equation}
Due to the regularity of $G_{\bm{x}}$, $H_{\bm{x}}$ belongs to $C^\infty(\overline{B_{R}(\bm{x})})$.
Moreover it also holds 
\begin{equation} \label{PropH}
H_{\bm{x},i}({\bm y}):=\frac{1}{\left| \tilde{F}_{\bm{x},i}(\bm{y}) \right|^2}  
\left( 2 \left\langle \tilde{F}_{\bm{x},i}(\bm{y}), G_{\bm{x},i}(\bm{y}) \right\rangle
+ \left| G_{\bm{x},i}(\bm{y}) \right|^2 \right) \ \text{ on }\  D_{\bm R}({\bm x})
\ \text{ and } \ H_{\bm{x},i}({\bm x})=0.
\end{equation}
From \eqref{EGsurF}, we deduce
\begin{equation} \label{EstH}
| H_{\bm{x},i}({\bm y})| < \frac{7}{9} \ \text{ on } \ B_{\bm R}({\bm x}) .
\end{equation}
With \eqref{e.Gbound}, it follows that the Neumann series
\begin{equation*} 
\frac{1}{1 + H_{\bm{x},i}(\bm{y})} = \sum_{j=0}^\infty (-1)^j H_{\bm{x},i}(y)^j,
\end{equation*}
converges in all $C^k(\overline{B}_{{\bm R}})$ spaces.
On the other side, as
\begin{equation} \label{2.13'}
\gamma_{i}^c \frac{ \tilde{F}_{\bm{x},i}(\bm{y})^\perp}{| \tilde{F}_{\bm{x},i}(\bm{y})|^2} = y_i-x_i, 
\end{equation}
we have
\begin{equation} \label{Jdiff}
J_{\bm{x},i}(\bm{y}) - x_i  =  y_i - x_i + |x_i-y_i|^2 G_{\bm{x},i}(\bm{y}) ^{\perp} \displaystyle  \frac{1}{1 + H_{\bm{x},i}({\bm y})}.
\end{equation}
We deduce with~\ref{e.Gbound} that $J_{\bm{x}}$ admits a $C^\infty$ extension $\overline{J}_{\bm{x}}$ on $B_{\bm{R}}(\bm{x})$, and that 
\begin{equation}
\overline{J}_{\bm{x}}(\bm{x}) = \bm{x} \ \text{ and } \
D\overline{J}_{\bm{x}}(\bm{x}) = \textrm{Id}
\end{equation}

\ \par	
\noindent
\textbf{4. Inverse function theorem.}
According to the inverse function theorem, there exist $\bm{r},\, \bm{r}^{\prime}$ in $(\R_+^*)^N$, and $W$ a neighborhood of $x$ with $B_{\bm{r}^{\prime}}(\bm{x}) \subset W$,  such that $\overline{J}_{\bm{x}}$ is a $C^\infty$ diffeomorphism from $B_{\bm{r}}(\bm{x})$ to $W$.
Moreover, thanks to \eqref{AbsTildeF}, one can obtain estimates for $D^2 \tilde{F}$ on $B_{\bm R}({\bm x})$ depending on $(\gamma_{i})_{i=1..N}$, $(\gamma_{i}^c)_{i=1..N}$ and $\min_{i \neq j} |x_{i} - x_{j}|$ only. Thus with the bound \eqref{e.Gbound} and the relations \eqref{DefH} and \eqref{ExprJ}, we can estimate $D^2 \overline{J}_{\bm{x}}$ on $B_{\bm R}({\bm x})$ with a dependence only on the previous quantities. As the minimal radius of such a neighborhood $W$ can be determined only relying on $\frac{1}{\| D^2 J_{\bm{x} \|} }$ (see for instance the version of the inverse function theorem given in \cite[Chapter 6, Lemma 1.3]{Lang}), one can choose  $W$ as a ball with a radius $r^{\prime}$ depending only on $(\gamma_{i})_{i=1..N}$, $(\gamma_{i}^c)_{i=1..N}$ and $\min_{j \neq i} |x_{i} - x_{j}|$. In the sequel we denote $\mathcal{V}_{\bm{x}} :=\overline{J}^{-1}_{\bm{x}}(D_{\bm{r}^{\prime}}(\bm{x}))$ the neighborhood of $\bm{x}$ such that $\overline{J}_{\bm{x}}$ is a $C^\infty$ diffeomorphism from $\mathcal{V}_{\bm{x}} \cup \{ \bm{x} \} $ to $B_{\bm{r}^{\prime}}(\bm{x})$.  \par
\ \par
\noindent	
\textbf{5. Conclusion for $F_{\bm{x}}$.}
According to the previous argument, $J_{\bm{x}}$  is a $C^\infty$ diffeomorphism from $\mathcal{V}_{\bm x}$ to $D_{\bm{r}^{\prime}}(\bm{x})$. Moreover Lemma~\ref{t.yexplicit.dimN} ensures the existence of $\bm{a}$ in $(\R_+^*)^N$, inversely proportional to $\bm{r}^{\prime}$, thus depending only on $(\gamma_{i})_{i=1..N}$, $(\gamma_{i}^c)_{i=1..N}$ and $\min_{j \neq i} |x_{i} - x_{j}|$, with $a_i$ increasing when $\min_{j \neq i} |x_{i} - x_{j}|$  decreases, such that $\tilde{F}_{\bm{x}}$ is a $C^\infty$ diffeomorphism from $D_{\bm{r}^{\prime}}(\bm{x}) $ to $Y_{\bm{a}}$.
As $F_{\bm{x}} = \tilde{F}_{\bm{x}} \circ J_{\bm{x}}$, we conclude that $F_{\bm{x}}$ realizes a $C^\infty$ diffeomorphism from $\mathcal{V}_{\bm x}$ to $Y_{\bm{a}}$. \par
\ \par
\noindent	
\textbf{6. Proof of \eqref{eq.iter}.}
As a consequence of the inverse mapping theorem, there exists a constant $K_a>0$ depending on $\bm{a}$ introduced above, thus on $(\gamma_{i})_{i=1..N}$, $(\gamma_{i}^c)_{i=1..N}$ and $\min_{j \neq i} |x_{i} - x_{j}|$, such that for all $\bm{\tilde{y}} \in \mathcal{V}_{\bm{x}}$ and $\bm{\tilde{u}} \in D_{\bm{r}^{\prime}}(\bm{x})$,
\begin{equation} \label{eq.contraction}
\left\|\bm{\tilde{y}} - J^{-1}_{\bm{x}}(\bm{\tilde{u}}) \right\| \leq K_a \left\|\bm{\tilde{u}} - J_{\bm{x}}(\bm{\tilde{y}})\right\|.
\end{equation}
For $\bm{v}$ in $Y_{\bm{a}}$ we let $\bm{u} := \tilde{F}^{-1}_{\bm{x}}(\bm{v})$. Hence $\bm{u} \in D_{\bm{r}^{\prime}}(\bm{x})$ and  $J_{\bm{x}}^{-1}(\bm{u})= F^{-1}_{\bm{x}}(\bm{v})$, thus putting $\bm{\tilde{u}}$ in \eqref{eq.contraction} leads to \eqref{eq.iter}.

\par
This ends the proof of Proposition~\ref{t.nvorticesautonomous}.
\end{proof}	
\begin{Rema}
We could consider the intensities $(\gamma_1,\ldots,\gamma_N)$ and $(\gamma_1^C,\ldots,\gamma_1^N)$ as variables. When these intensities are fixed, Proposition~\ref{t.nvorticesautonomous} proves that a large enough $v \in (\R^{2})^N$ is attained by $F_{\bm x}$. Then a simple scaling argument shows that any $v \in (\R^{2})^N$ with no zero-component is attained provided that the intensities $(\gamma_1,\ldots,\gamma_N)$ and $(\gamma_1^C,\ldots,\gamma_1^N)$ are sufficiently small.
\end{Rema}
\ \par
We are now in position to prove Theorem~\ref{t.NvortexNcontrols}.
\begin{proof}[Proof of Theorem~\ref{t.NvortexNcontrols}]
This proof consists in three steps. First, we prove the existence of a certain family of $N$ curves from $x_{0,1},\dots, x_{0,N}$ (close to some fixed $\overline{x}_{0,1},\dots, \overline{x}_{0,N}$) to $x_{f,1},\dots, x_{f,N}$, and satisfying properties compatible with Proposition~\ref{t.nvorticesautonomous}.  In a second part, assuming that the initial position of the controls $\bm{y}$, say $\bm{y}(0) = (y_{0,1},\dots, y_{0,N})$ is close enough to $(x_{0,1},\dots, x_{0,N})$, we show how to construct $\bm y$ so that the vortices $\bm x$ follow the prescribed trajectories. The last step explains how to reduce to the previous situation when the assumption on the initial positions of the controls is not satisfied. \par
\ \par
\noindent
\textbf{1. Construction of a family of curves.}	
We first define a family of reference curves in the following lemma.
\begin{Lemm} \label{LemCourbes}
Given $N$ distinct points $\overline{x}_{0,1}$, \dots, $\overline{x}_{0,N}$ in $\R^2$ and distinct points $x_{f,1}$, \dots, $x_{f,N}$ in $\R^2$, there exists $r>0$, such that for any $T> 0$ and any $v_{\min}>0$, any $\tilde{x}_{0,1} \in B(\overline{x}_{0,1};r)$, \dots, $\tilde{x}_{0,N} \in B(\overline{x}_{0,N};r)$, any $v_{0,1},\ldots,v_{0,N}$ in $\R^2$ such that for all $i=1,\ldots,N$, $|v_{0,i}| \geq v_{\min}$, one can find $C^{\infty}$ curves $\Gamma_{1}, \dots, \Gamma_{N}: [0,T] \rightarrow \R^2$ satisfying
\begin{gather}
\label{Courbes0T}
\forall i \in \{1, \dots, N \}, \ \Gamma_{i}(0)=\tilde{x}_{0,i} \ \text{ and } \ \Gamma_{i}(T)=x_{f,i}, \\ 
\label{CourbesSeparees}
\forall t \in [0,T], \ \forall i \neq j, \  \ |\Gamma_{i}(t) - \Gamma_{j}(t)| \geq r, \\
\label{CourbesVmin}
\forall t \in [0,T], \forall i \in \{1, \dots, N \},  \ |\dot{\Gamma}_i(t)| \geq v_{\min}, \\
\label{Courbes0Vitesse}
\dot{\Gamma}_{i}(0)=v_{0,i} \ \text{ for all } \ i=1,\ldots,N.
\end{gather}
\end{Lemm}
\begin{proof}[Proof of Lemma~\ref{LemCourbes}]
We start with the case $x_{i}=\overline{x}_{i}$ for all $i=1\ldots N$, and without the constraint \eqref{Courbes0Vitesse}. \par
In that case , 
it is easy to construct smooth and simple curves $\overline{\mathcal{C}}_{1}, \dots, \overline{\mathcal{C}}_{N}: [0,1] \rightarrow \R^2$, with disjoint graphs, with $|\dot{\overline{\mathcal{C}}}_{i}|>0$, and driving $\overline{x}_{0,i}$ to $x_{f,i}$. Indeed one first constructs $\overline{\mathcal{C}}_{1}$ by noticing that $\R^2 \setminus \{\overline{x}_{0,2}, \dots, \overline{x}_{0,N}, x_{f,2}, \dots, x_{f,N}\}$ is path-connected, and then one constructs $\overline{\mathcal{C}}_{2}$ by noticing that 
$\R^2 \setminus \left(\overline{\mathcal{C}}_{1}([0,1]) \cup \{\overline{x}_{0,3}, \dots, \overline{x}_{0,N}, x_{f,3}, \dots, x_{f,N}\}\right)$
is path-connected, etc. Moreover, one can ask that in small neighborhood of $t=0$, $\dot{\overline{\mathcal{C}}}_{i}(t)$ is constant. \par
Now we let
\begin{equation*} 
r:= \frac{1}{4} \min \{ |\overline{\mathcal{C}}_{i}(s) - \overline{\mathcal{C}}_{j}(t)| , \ s,t \in [0,1], \ i \neq j \} >0.
\end{equation*}
In order to achieve \eqref{CourbesVmin}, we introduce a slight modification of the curves $\overline{\mathcal{C}}_{1}, \dots, \overline{\mathcal{C}}_{N}$ . We consider small circles passing through $x_{f,1}$, \dots, $x_{f,N}$, with diameter less than $r/4$, and parameterized by $C^{\infty}$ mappings $c_{1}, \dots, c_{N}:\R \rightarrow \R^2$, $1$-periodic with $c_{i}(0)=x_{f,i}$ and $\dot{c}_{i}(0) \perp \dot{\overline{\mathcal{C}}}_{i}(1)$ for $i=1,\ldots,N$. We let $\varphi \in C^\infty(\R; \R)$ a decreasing function such that $\varphi=1$ on $(\infty,-1]$ and $\varphi=0$ on $[0,+\infty)$. Then one set the curves $\overline{\Gamma}_{1},\dots, \overline{\Gamma}_{N}$ given by the following formula for $k$ and $n$ large enough (assuming $\overline{\mathcal{C}}_{i} (t)= x_{f,i}$ for $t>1$):
\begin{equation} \label{Curveconstruc} 
\overline{\Gamma}_{i}(t) := 
\varphi\left(n \left(t-\frac{T}{k}\right) \right) \overline{\mathcal{C}}_{i}\left(\frac{kt}{T}\right) 
 + \left[1- \varphi \left(n\left(t-\frac{T}{k}\right)\right)\right] c_{i}\left(\frac{kt}{T}\right).
\end{equation}
One checks that condition \eqref{CourbesVmin} is satisfied for $k$ large enough as $\dot{\overline{\Gamma}}_i(t) \neq 0$ for $t$ in $[0,T]$, and condition \eqref{CourbesSeparees} is satisfied for $n$ large enough.

It remains to explain how to treat other starting points $\tilde{x}_{0,i} \in B(\overline{x}_{0,i};r)$ and to obtain \eqref{Courbes0Vitesse}. The idea is to go from $\tilde{x}_{0,i}$ to $\overline{x}_{0,i}$ in a very short time, and then to follow the previous trajectory $\overline{\Gamma}_i$. To do so, we introduce $C^{\infty}$ curves $\tilde{c}_{1},\dots, \tilde{c}_{N}:[0,\varepsilon] \rightarrow \R^2$ for $\varepsilon>0$ small enough, such that in a neighborhood of $t=0$, $\tilde{c}_{i}(t)= \tilde{x}_{0,i} + v_{0,i} t$, in a neighborhood of $t=\varepsilon$, $\tilde{c}_{i}(t)= \overline{x}_{0i} + \dot{\overline{\Gamma}}_{i}(0)(t-\varepsilon)$, during $[0,\varepsilon]$, one has $\tilde{c}_{i} \in B(\overline{x}_{i,0};r)$ and $|\dot{\tilde{c}}_{i}| \geq v_{\min}$. This is easily obtained for small $\varepsilon$. Then we finally rescale in time by setting for $i \in \{1,\dots, N\}$,
\begin{equation*} 
\Gamma_{i}(t)= \tilde{c}_{i}\left(\frac{T+\varepsilon}{T}t\right) \ \text{ in } \ \left[0, \frac{\varepsilon T}{T+\varepsilon}\right]
 \ \text{ and }\
\Gamma_{i}(t)= \overline{\Gamma}_{i} \left(\frac{T + \varepsilon}{T} t- \varepsilon\right) \ \text{ in } \ \left[\frac{\varepsilon T}{T+\varepsilon}, T\right],
\end{equation*}
which is $C^{\infty}$ since $\dot{\overline{\Gamma}}_i(t)$ is constant in a neighborhood of $t=0$. \par
This ends the proof of Lemma~\ref{LemCourbes}.
\end{proof}
\ \par
\noindent
%
%
\textbf{2. Construction of the control when $\bm{y}_0$ is close to $\bm{x}_0$.}
In this step, we first consider the case where $\bm{y}_0$ close to $\bm{x}_0$; the general case will be deduced later. \par
Given $x_{0,1}, \dots, x_{0,N}$ and $x_{f,1}, \dots, x_{f,N}$, we use Lemma~\ref{LemCourbes} with $(\overline{x}_{0,1}, \dots, \overline{x}_{0,N})=(x_{0,1}, \dots, x_{0,N})$ as a base point. Then according to the first part of the proof there exists a minimal distance $r>0$ between the possible curves $\Gamma_1, \dots, \Gamma_N$, for now independently of $v_{\min}>0$, of $\tilde{x}_{0,1} \in B(\overline{x}_{0,1};r), \dots, \tilde{x}_{0,N} \in B(\overline{x}_{0,N};r)$, and of $v_{0,1},\ldots,v_{0,N}$ in $\R^2$ (such that for all $i=1,\ldots,N$, $|v_{0,i}| \geq v_{\min}$.) \par
Recalling Proposition~\ref{t.nvorticesautonomous} and in particular the monotonicity of $\bm{a}$ with respect to $\min_{i \neq j } |\Gamma_i(t) - \Gamma_j(t)|$, there exists $\bm{a} \in (\R^*_+)^N$ corresponding to $\min_{t \in [0,T]} \min_{i \neq j } |\Gamma_i(t) - \Gamma_j(t)|$ such that for all $t \in [0,T]$ $F_{\bm{\Gamma}(t)}$ is surjective onto $Y_{\bm{a}}$. Then we fix
\begin{equation} \label{DefVmin}
v_{\min}=\max(a_{1},\ldots,a_{N}).
\end{equation}
Now we suppose that for all $i \in 1..N$, $y_{0,i}$ is sufficiently close to $\tilde{x}_{0,i}$ for the condition $|F_{\bm{\tilde{x}_0},i}(\bm{y}_0) | > v_{\min}$ to be verified. We express this condition with a parameter $\hat{r}$:
\begin{equation} \label{HypYprocheX}
| y_{0,i} - \tilde{x}_{0,i} | < \hat{r},
\end{equation}
 and we notice that $\hat{r}$ merely depends on $r$ and $v_{\min}$.  \par
Now given  $\tilde{x}_{0,1} \in B(\overline{x}_{0,1};r)$, \dots, $\tilde{x}_{0,N} \in B(\overline{x}_{0,N};r)$, we deduce completely the curves $\Gamma_{1}, \dots , \Gamma_{N}$ according to the minimal speed $v_{\min}$ and by setting
\begin{equation*} 
v_{0,i} :=  F_{\tilde{\bm{x}}_{0},i}(\bm{y}_{0}) .
\end{equation*}
We then define, for $t$ in $[0,T]$,
\begin{equation*}
\bm{v}(t) := \frac{d\bf{ \Gamma }}{dt}(t),
\end{equation*}
Hence according to Proposition~\ref{t.nvorticesautonomous}, we may now define the control as following:
\begin{equation} \label{ZeControl}
\forall t \in \left[0,T\right],\, \bm{y}(t)= F^{-1}_{\bf \Gamma(t) }(\bm{v}(t)).
\end{equation}
Thanks to the regularity of $\bm{v}$, $\bm{y} \in C^{\infty}(\left[0,T\right],\R^{2})^N$, and the corresponding trajectory leads $\tilde{\bm{x}}_0$ to $\bm{x_f}$ in time $T$ by construction. \par
\ \par
\noindent
\textbf{3. Starting from any ${\bm y}_{0}$.} We now explain how to reduce the general case to the case where the assumption \eqref{HypYprocheX} is satisfied. The idea is to introduce a first step during which we bring all $y_{i}$ close to $x_{0,i}$ in a \textit{very short time}. This will slightly affect the position of the vortices $x_{i}$, but  nevertheless ensure that condition \eqref{HypYprocheX} is fulfilled. \par
Given $x_{0,1}$, \dots, $x_{0,N}$ and $x_{f,1}$, \dots, $x_{f,N}$, we apply step 2. 
We deduce some $r>0$ and $\hat{r}>0$ such that we know how to drive  $\tilde{x}_{0,1} \in B(x_{0,1};r)$, \dots, $\tilde{x}_{0,N} \in B(x_{0,N};r)$ to $\bm{x}_{f}$ provided that \eqref{HypYprocheX} is satisfied. \par
Now arguing as in the first part of this proof with a path-connectedness argument, we construct curves $\mathfrak{C}_{1},\ldots,\mathfrak{C}_{N}:[0,1] \rightarrow \R^2$ joining $(y_{0,1}, \dots, y_{0,N})$ to some points $(\check{x}_{0,1}, \dots, \check{x}_{0,N})$ in $S(x_{0,1},\hat{r}/2) \times \cdots \times S(x_{0,N},\hat{r}/2)$, such that these curves and the points $x_{0,i}$ stay at a minimal distance $\tilde{r}>0$, that is to say
\begin{align*} 
\min \{ |\mathfrak{C}_{i}(s) - x_{0,j}|, \ s \in [0,1], \ i,j =1\ldots N \} & > \tilde{r}, \\
\min \{ |\mathfrak{C}_{i}(s) - \mathfrak{C}_{j}(t)| , \ s,t \in [0,1], \ i \neq j \} & > \tilde{r}, \\
\min\{|x_{0,i} - x_{0,j}|, i \neq j\} &> \tilde{r} .
\end{align*}
Then for $\varepsilon>0$ small, we consider the solution $\bm{x}^\varepsilon:=(x_{1}^{\varepsilon}, \dots, x_{N}^{\varepsilon})$ of \eqref{e.odeN} for $t \in [0,\varepsilon]$, with the control $\bm{y}^{\varepsilon}$ given by
\begin{equation*} 
y^\varepsilon_{i}(t) = \mathfrak{C}_{i} \left(\frac{t}{\varepsilon}\right) \ \text{ for } \ t \in [0,\varepsilon], \ i =1,\ldots,N.
\end{equation*}
Let us show the existence of $\varepsilon_0$ such that for all $\varepsilon < \varepsilon_0$, $\bm{x}^{\varepsilon}$ is well defined on $[0,\varepsilon]$, namely that the trajectories of the vortices and the controls do not cross. 
This is a continuous induction argument. \par
At initial time, by construction, for $i \neq j$, we have $|x_{0,i}-x_{0,j}| > \tilde{r}$ and $|x_{0,i}-y_{0,j}| > \tilde{r}$. hence by continuity of $\bm{x}^{\varepsilon}$, for each $\varepsilon >0$ there exists $t_{\varepsilon} > 0$ such that for $t \in [0,t_{\varepsilon}]$:
\begin{equation}\label{Xeps}
|x_{i}^{\varepsilon}(t)-x_{j}^{\varepsilon}(t)| > \frac{\tilde{r}}{2} \text{ and } |x_{i}^{\varepsilon}(t)-y_{j}^{\varepsilon}(t)| > \frac{\tilde{r}}{2}.
\end{equation}
It follows that on $[0,t_{\varepsilon}]$,
\begin{equation} \label{Vepsmax}
\left|\frac{d x_i^{\varepsilon}}{dt} \right|(t) < \sum_{j \neq i} \frac{2|\gamma_j|}{\tilde{r}} + \sum_{j=1..N} \frac{2|\gamma^c_j|}{\tilde{r}} =: |v_i|_{\infty}. 
\end{equation}
Since $|v_i|_{\infty}$ does not depend on $\varepsilon$, we deduce that \eqref{Xeps} holds for all $t \in [0,\varepsilon]$ for $\varepsilon < \varepsilon_0 := \frac{\tilde{r}}{4 \max |v_i|_{\infty} }$. Thus $\bm{x}^{\varepsilon}$ is well defined for $\varepsilon$ sufficiently small. 
Moreover according to \eqref{Vepsmax}, 
\begin{equation*} 
\| \bm{x}^\varepsilon - \bm{x}_{0} \|_{L^{\infty}(0,\varepsilon)} \leq \varepsilon \max_{1\leq i \leq N} |v_i|_{\infty}
\longrightarrow 0 \ \text{ as } \ \varepsilon \rightarrow 0^+.
\end{equation*}
Hence for suitably small $\varepsilon$, $\bm{x}^\varepsilon(\varepsilon)$ belongs to $B(x_{0,1},r) \times \dots \times B(x_{0,N},r)$ and one can then apply the second step of this proof with $\bm{x}^\varepsilon(\varepsilon)$ as initial points,  $\bm{y}^\varepsilon(\varepsilon)$ as initial positions of the control vortices and $T-\varepsilon$ as a time horizon. It suffices indeed to rescale in time the curves $\Gamma_{i}$, which does not affect the construction as long as we go {\it faster} (see in particular how $\hat{r}$ is defined after \eqref{HypYprocheX}). Moreover, proceeding as previously, we may ensure that the connection of the two parts of the solution (during $[0,\varepsilon]$ and $[\varepsilon,T]$) is of class $C^\infty$. \par
\ \par
This ends the proof of Theorem~\ref{t.NvortexNcontrols}.
\end{proof}
\subsection{A variant of Theorem~\ref{t.NvortexNcontrols}}
Let us finish with a variant of Theorem~\ref{t.NvortexNcontrols} obtained as a consequence of Variant~\ref{Var:onevortex}. It allows, when ${\bm x}_{0}$ and ${\bm x}_{f}$ are sufficiently close in a certain sense, to ensure that  the vortices $x_{i}$ according to the previous construction are moving along straight lines. It enables to have a clear localization of all vortices $x_{i}$ and $y_{i}$ at all times. \par
\begin{Theo} \label{Var:Nvortex}
Let $\overline{\bm{x}}_f= (\overline{x}_{f,1},\dots, \overline{x}_{f,N}) $ in $(\R^{2})^N$ with $\overline{x}_{f,i}\ne \overline{x}_{f,j}$ for $i,j$ in $\{1,\ldots,N\}$ and $i\ne j$.
There exists $ D_{0} \in \big(0, \min_{i \neq j} |\overline{x}_{f,i} - \overline{x}_{f,j}|/8 \big)$ such that for any $D \in (0,D_{0})$, setting
\begin{equation} \label{DefTauRho}
\tau = \frac{D^2}{\min_{i} |\gamma^c_{i}|} \ \text{ and } \ \rho = \min \left(\frac{D}{8},D^3\right),
\end{equation}
the following holds. Let $\overline{\bm{x}}_0= (\overline{x}_{0,1},\dots, \overline{x}_{0,N}) $ in $(\R^{2})^N$ such that for all $i$, $\overline{x}_{0,i} \in \overline{B}(\overline{x}_{f,i},D) \setminus \overline{B}(\overline{x}_{f,i},D/2)$. Let
\begin{itemize}
\item ${\mathcal X}_{i}$ the ``ice-cream cone'' $\mbox{Conv} \left(\{ \overline{x}_{0,i} \} \cup \overline{B}(\overline{x}_{f,i}, \rho)\right)$,
\item ${\mathcal Y}_{i}$ the ``stadium'' $\gamma^c_{i} \tau \displaystyle \frac{(\overline{x}_{f,i} - \overline{x}_{0,i})^\perp}{|\overline{x}_{f,i} - \overline{x}_{0,i}|^2} + \big[\overline{x}_{0,i}, \overline{x}_{f,i}\big] +  \overline{B}(0, D/8)$.
\end{itemize}
Then all ${\mathcal X}_{i}$ and ${\mathcal Y}_{i}$ are disjoint compact convex sets and moreover,
for any $\bm{x}_{f}= (x_{f,1}, \dots, x_{f,N})$ in $\overline{B}(\overline{x}_{f,1},\rho) \times \cdots \times \overline{B}(\overline{x}_{f,N},\rho)$,
one can find  controls $y_1, \dots, y_N : [0,\tau] \rightarrow \R^2$ with values in ${\mathcal Y}_{1}, \ldots, {\mathcal Y}_{N}$ respectively,
driving $\overline{{\bm x}}_{0}$ to ${\bm x}_{f}$ in time $\tau$ and such that the corresponding solutions $x_1, \ldots, x_N$ of \eqref{e.odeN} are straight lines belonging to ${\mathcal X}_{1}, \ldots, {\mathcal X}_{N}$ respectively.
\end{Theo}
\begin{proof}
The proof is composed of three steps.
We first give a first sufficient condition on $D$ (characterizing the admissibility set for $\overline{\bm{x}}_0$) such that it is possible to construct controls $\bm{y}$ associated to straight lined trajectories between $\overline{\bm{x}}_0$ and $\bm{x}_f$.
Then by comparing $\bm{y}$ and the control $\bm{\tilde{y}}$ corresponding to the "uncoupled" system, we show that for $D$ small enough (and with $\rho$ and $\tau$ defined as above), $y_i$ belongs to $\mathcal{Y}_i$ for all $i$. Finally, one shows the disjunction of the sets $\mathcal{X}_i$, $\mathcal{Y}_i$ for all $i,j$ by using again the smallness of $D$ and the form of $\tau$ and $\rho$.

\ \par
\noindent
\textbf{1. Construction of the controls associated to straight-lined trajectories.}
This step of the proof focuses on sufficient conditions to ensure the invertibility of the right-hand side of the system~\eqref{e.odeN}.
We first apply Proposition~\ref{t.nvorticesautonomous} ensuring the monotonicity of $\bm{a}$ with respect to $\min_{i \neq j } |x_i - x_j|$. We can therefore introduce $\bm{a} \in (\R^*_+)^N$ such that $F_{\bm{x}}$ is surjective onto $Y_{\bm{a}}$ for all ${\bm x}=(x_{1}, \ldots, x_{n}) \in (\R^2)^N$ satisfying
\begin{equation} \label{XiXjLoin}
|x_{i} - x_{j}| \geq \frac{\min_{i \neq j} |\overline{x}_{f,i} - \overline{x}_{f,j}|}{2}.
\end{equation}
 We deduce $v_{\min}$ given by \eqref{DefVmin}, and $K_a$ so that \eqref{eq.iter} applies. \par
Now, assuming at first $D \in \big(0, \min_{i \neq j} |\overline{x}_{i,f} - \overline{x}_{j,f}|/8 \big)$ we construct the trajectories as follows.
Given ${\bm x}_{f}$ in $\overline{B}(\overline{x}_{f,1},\rho) \times \cdots \times \overline{B}(\overline{x}_{f,N},\rho)$, we introduce  $\Gamma_i: [0,\tau] \mapsto \R^2$ as straight lines from $\overline{x}_{0,i}$ to ${x}_{f,i}$:
\begin{equation*}
\Gamma_{i}(t) = \overline{x}_{0,i} + \frac{t}{\tau}(x_{f,i}-\overline{x}_{0,i}), \ \ \forall t \in [0,\tau].
\end{equation*}
These trajectories clearly satisfy \eqref{XiXjLoin} for all $t \in [0, \tau]$, as $\Gamma_i(t) \in B(\overline{x}_{fi},D)$ for all $t \in [0, \tau ]$. \par
Then following Variant~\ref{Var:onevortex}, we can construct for all $i \in \{1,\dots,N\}$ the associated ``uncoupled'' control vortices
\begin{equation} \label{DefYtildei}
\tilde{y}_i(t) := \tilde{F}^{-1}_{\bm{\Gamma}(t),i} (\bm{\dot{\Gamma}}(t)) = \Gamma_{i}(t) + \gamma^c_{i} \tau \frac{(x_{f,i}-\overline{x}_{0,i})^\perp}{|x_{f,i}-\overline{x}_{0,i}|^2},
\end{equation}
and the ``coupled'' ones
\begin{equation*}
{\bm y}(t) := F_{{\bm \Gamma(t)}}^{-1} \left(\dot{\Gamma}_{1}(t),\ldots,\dot{\Gamma}_{N}(t)\right)
= F_{{\bm \Gamma(t)}}^{-1} \left( \frac{x_{f,1}-\overline{x}_{0,1}}{\tau} ,\ldots, \frac{x_{f,N}-\overline{x}_{0,N} }{\tau} \right).
\end{equation*}
This last construction is possible if 
\begin{equation} \label{CondVitesse}
	\left\lVert \displaystyle \frac{\bm{x_{f}} - \bm{\overline{x}_{0}}}{\tau}  \right\rVert_{\infty} \geq v_{\text{min}}. 	
\end{equation}
As for all $i\in 1,\ldots,N$, we have $\overline{x}_{0,i} \in \overline{B}(\overline{x}_{f,i},D) \setminus \overline{B}(\overline{x}_{f,i},D/2)$ and $x_{f,i} \in \overline{B}(\overline{x}_{f,i},\rho)$, we deduce
\begin{equation} \label{xfx0min}
|x_{f,i} - \overline{x}_{0,i}| \geq \frac{3}{8} D. 
\end{equation}
With the definition of $\tau$ given in \eqref{DefTauRho}, \eqref{CondVitesse} amounts to a second smallness condition on $D$ (that yields $D_0$):
\begin{equation} \label{Eq:CondD1}
D \leq \frac{3}{8} \frac{\min_{i} |\gamma^c_{i}|}{v_{\min}}.
\end{equation}
It is straightforward that $\bm{y}$ drives $\overline{{\bm x}}_{0}$ to ${\bm x}_{f}$ in time $\tau$ and and that the corresponding solutions $x_1, \ldots, x_N$ of \eqref{e.odeN} are straight lines belonging to ${\mathcal X}_{1}, \ldots, {\mathcal X}_{N}$ respectively. 

\ \par
\noindent
\textbf{2. Location of the controls.}
Let us now show that the controls $y_{i}$ belong to ${\mathcal Y}_{i}$ for $\rho$ given as in \eqref{DefTauRho}. 
We first introduce the trajectories $\bm{\underline{x}}$ and the ``uncoupled'' controls $\bm{\underline{y}}$ corresponding to the particular case ${\bm x}_{f} = \overline{\bm x}_{f}$:
\begin{equation} \label{DefUnderlineX}
\underline{x}_{i}(t) := \overline{x}_{0,i} + \frac{t}{\tau}(\overline{x}_{f,i}-\overline{x}_{0,i}) 
\ \text{ and } \ 
\underline{\tilde{y}}_i(t) :=  \underline{x}_{i}(t) + \gamma_{i}^c \tau \frac{(\overline{x}_{f,i}-\overline{x}_{0,i})^\perp}{|\overline{x}_{f,i}-\overline{x}_{0,i}|^2},
\end{equation}
corresponding to the ``middle lines'' in the ice cream cone ${\mathcal X}_{i}$ and in the stadium ${\mathcal Y}_{i}$, respectively. We now locate $y_{i}$ by means of
\begin{equation} \label{LocalisationY}
|y_{i}(t) - \underline{\tilde{y}}_{i}(t)| \leq 
|y_{i}(t) - \tilde{y}_{i}(t)| + | \tilde{y}_{i}(t) - \underline{\tilde{y}}_{i}(t)|.
\end{equation}
Let us now estimate the two terms in the right-hand side of \eqref{LocalisationY}.

Concerning the second term, we use that $x \mapsto \frac{x^\perp}{\lvert x\rvert^2}$ can be represented by the complex map $z \mapsto \frac{i}{\overline{z}}$ for $z = x_1 + i x_2$ and that $\left|\frac{1}{z} - \frac{1}{z'}\right| = \frac{|z-z'|}{|zz'|}$ in $\R^2$. Using \eqref{DefYtildei}, \eqref{xfx0min} and \eqref{DefUnderlineX}, and noticing that $\vert \underline{x}_{i}(t) - \Gamma_i(t)\vert \leq \rho$ on $[0,\tau]$, we find that
\begin{equation*}
| \tilde{y}_{i}(t) - \underline{\tilde{y}}_{i}(t)| \ \leq \  \rho + |\gamma_{i}^c| \tau \frac{64}{9D^2} \rho  
\ \leq \  D^3 \left( 1+ \frac{64}{9} \frac{|\gamma_{i}^c|}{\min_{j} |\gamma^c_{j}|} \right).
\end{equation*}
For $D$ sufficiently small, this is clearly less than $D/16$. 
\par
Concerning the first term in the right-hand side of \eqref{LocalisationY}, according to \eqref{eq.iter} and recalling \eqref{DefJ}, we have for all $t \in [0,\tau]$, 
\begin{equation} \label{eq.ineqrefcontrol}
\left\| \bm{\tilde{y}}(t) - \bm{y}(t) \right\| \leq K_a \left\| \tilde{\bm y}(t) - J_{\bm {\Gamma}(t)} (\tilde{\bm y}(t)) \right\|.
\end{equation}
Moreover due to \eqref{e.Gbound} and \eqref{Jdiff}, we have for some constant $C>0$ depending on $(\gamma_{i})_{i=1..N}$, $(\gamma_{i}^c)_{i=1..N}$ and ${\bm x}_{f}$ only,
\begin{equation}
\left\| \tilde{\bm{y}}(t) - J_{\bm{\Gamma}(t)} (\tilde{\bm y}(t)) \right\|
\leq C \| \tilde{\bm y}(t) - \bm{\Gamma}(t)\|^2 \leq C \frac{\max\vert \gamma_i^c \vert^2 \, \tau^2}{\| {\bm x}_{f} - {\bm x}_{0} \|^2}.
\end{equation}
Again, this is less than $D/16$ for suitably small $D$.
Going back to \eqref{LocalisationY}, this gives $y_{i}(t) \in {\mathcal Y}_{i}$ in all configurations. \par
\ \par
\noindent
\textbf{3. Disjunction of the sets ${\mathcal X}_{i}$ and ${\mathcal Y}_{j}$.}
It remains to show that all the ${\mathcal X}_{i}$ and ${\mathcal Y}_{j}$ are disjoint. To do so, we estimate:
\begin{equation*}
\left|\underline{\tilde{y}}_{i}(t) - \underline{x}_{i}(t) \right| 
= \frac{|\gamma_{i}^c|\tau}{|\overline{x}_{f,i}-\overline{x}_{0,i}|} \geq  \frac{|\gamma_{i}^c|\tau}{D}.
\end{equation*}
The choice \eqref{DefTauRho} of $\tau$ makes this term larger than or equal to $D$. Hence we deduce that $\mbox{dist}({\mathcal X}_{i},{\mathcal Y}_{i}) \geq \frac{3D}{4}$, for all $i$. Finally, since in the other direction
\begin{equation*}
\left|\underline{\tilde{y}}_{i}(t) - \underline{x}_{i}(t) \right| 
= \frac{|\gamma_{i}^c|\tau}{|\overline{x}_{f,i}-\overline{x}_{0,i}|} \leq 2 \frac{|\gamma_{i}^c|\tau}{D} = \frac{2 |\gamma_{i}^c| }{\min_{j} |\gamma^c_{j}|} D ,
\end{equation*}
we deduce that if $D$ is small enough (depending on ${\bm x}_{f}$), then for all $i$, \[{\mathcal X}_{i} \cup {\mathcal Y}_{i} \subset \overline{B}(\overline{x}_{i,f}, \min_{j \neq k} |\overline{x}_{j,f} - \overline{x}_{k,f}|/4). \] This is enough to ensure that all the sets ${\mathcal X}_{i} \cup {\mathcal Y}_{i}$, $i=1,\ldots,N$, are disjoint.

\end{proof}
%
%
%
%
%
%
%
%
\section{Reduction to a single control vortex}
\label{s.singlecontrol}
In this section, we establish Theorem~\ref{Thm:Main}. Hence we consider the control of the vortex system by means of a single control vortex $z: [0,T] \longrightarrow \R^2$, of intensity $\gamma^c \neq 0$. The evolution of the position of the vortex $x_i$, for $i$ in $\{1,\dots,N\}$, is now governed by System~\eqref{e.intromain}. Incorporating the factor $\frac{1}{2\pi}$ in the intensities $\gamma_{1},\ldots,\gamma_{N},\ \gamma^c$, and using $K$ the Biot-Savart kernel \eqref{Biot-Savart}, System~\eqref{e.intromain} is written as
\begin{equation}\label{e.ode1control}
\left\{
\begin{array}{rcl}
\displaystyle  \frac{dx_i}{dt}(t) &=& \displaystyle  \sum_{j \neq i} \gamma_j K* \delta_{x_j(t)}(x_i(t)) + \gamma^c K* \delta_{z(t)}(x_i(t)), \\
x_i(0) &=& x_{0,i}.
\end{array}
\right.
\end{equation}
The strategy is as follows. First, we introduce a reference solution $\bm{\overline{x}}$ with $N$ controls obtained by means of Theorem~\ref{t.NvortexNcontrols}.
Then we consider a solution of \eqref{e.ode1control} with an oscillating control that mimics the action of these $N$ controls. The proof consists then in comparing these two solutions a prove that they are suitably close for a sufficiently fast oscillating control. This will lead us to an approximate controllability result. Finally, we establish a local exact controllability result following the same lines (however with a simplified reference solution) and using a topological argument.

\subsection{A single oscillating control} \label{s.defcontrol}
We fix $T>0$, $\bm{x_0} \in (\R^{2})^N$, $\bm{x_f} \in (\R^2)^N$ and $y_{0}$ the initial position of the control vortex.  \par
\ \par
\noindent
\textbf{1. Reference solutions and controls.}
To begin with, we apply Theorem~\ref{t.NvortexNcontrols} with $N$ control vortices of intensity $\gamma^c/N$, where $\gamma^c$ is the intensity of the single control $z$ in \eqref{e.ode1control}. These $N$ control vortices are initially placed at $y_{0}$ for the first one, and arbitrarily away from $x_{0,1}, \ldots, x_{0,N}$, $y_{0}$ and one from another for the others. Hence we obtain reference controls $y_i(t)$, $i$ in $\{1,\ldots,N\}$ and reference solutions $\bm{\overline{x}}$ of \eqref{e.odeN} that go from $\bm{x_0}$ at time $0$ to $\bm{x_f}$ in time $T$.
These reference solutions satisfy, for $i \in \{1,\dots,N\}$,
\begin{equation} \label{e.ODExav}
\left\{
\begin{array}{rcl}
\displaystyle  \frac{d\overline{x}_i}{dt}(t) &=&  \displaystyle  \sum_{j \neq i} \gamma_j K* \delta_{\overline{x}_j(t)}(\overline{x}_i(t)) + \frac{\gamma^c}{N} \sum_{k=1}^N K* \delta_{y_k(t)}(\overline{x}_i(t)), \\
\overline{x}_i(0) &=& x_{0,i}.
\end{array}
\right.
\end{equation}
Note that Theorem~\ref{t.NvortexNcontrols} ensures that one can find $\overline{r}>0$ such that for all $t \in [0,T]$, 
\begin{equation} \label{e.ymin}
\min_{j,k=1\ldots N}|y_k(t) - \overline{x}_j(t)| \geq 2\overline{r} \ \text{ and }\ 
\min_{i\neq j}|\overline{x}_i(t) - \overline{x}_j(t)| \geq 4 \overline{r}.
\end{equation}
In the next step we introduce a sequence of controls $(z_n)_{n\in \N^*}$ oscillating between these $N$ reference control trajectories $(y_1,\dots, y_N)$. This follows the ideas of Filippov's convex integration \cite{Filippov:1967}; see also for instance \cite{Bressan:2019} regarding the controllability of a flock of animals by a repelling agent for similar ideas. \par
\ \par
\noindent
\textbf{2. General form of the control.}
Let us give a precise definition of this oscillating control on $[0,T]$.
For $n$ in $\N^*$, we divide the interval $I:=[0,T]$ in $n$ intervals of size

\begin{equation} \label{DeltaT}
\Delta T := \frac{T}{n},
\end{equation}
that is, we set
\begin{equation*}
I_i := [t_{i-1},t_i]=\left[ (i-1)\Delta T, i\Delta T \right] \ \text{ for } i \text{ in } \{1,\ldots,n\}.
\end{equation*}
Each interval $I_i$ in then subdivided in $N$ subintervals of size
\begin{equation} \label{deltaT}
\delta T := \frac{\Delta T}{N} = \frac{T}{nN},
\end{equation}
and we define accordingly
\begin{equation*}
I_{i,k} := [t_{i-1,k-1},t_{i-1,k}] = \left[ (i-1)\Delta T + (k-1) \delta T , (i-1)\Delta T + k \delta T \right]
\ \text{ for } k \text{ in } \{1,\ldots,N\}.
\end{equation*}
The general idea would be to consider the control
\begin{equation} \nonumber
\widehat{z}_n(t) := \sum_{i=1}^n \sum_{k=1}^N \mathds{1}_{I_{i,k}}(t) y_k(t)  \text{ for } t \in [0,T].
\end{equation}
However, to achieve a continuous trajectory, we split each of the intervals $I_{i,k}$ in two parts:
\begin{equation*}
I_{i,k}^w := \left[t_{i-1,k} - \frac{\delta T}{2n},t_{i-1,k} + \frac{\delta T}{2n} \right] \ \text{ and } \ 
I_{i,k}^c := \left[t_{i-1,k-1}+\frac{\delta T}{2n} ,t_{i-1,k} - \frac{\delta T}{2n}\right],
\end{equation*}
where $I_{i,k}^c$ is used to let the control at $y_k$, and  $I_{i,k}^w$ aims at ensuring a continuous transition between $y_k$ and $y_{k+1}$.
Moreover we set $I_{1,1}^c := \left[0, t_{0,1} - \frac{\delta T }{2n} \right], I_{n,N}^c :=  \left[t_{n-1,N-1}+\frac{\delta T}{2n} , T\right]$ and $I_{n,N}^w := \emptyset$, so that we have $I = \displaystyle \bigcup_{i=1}^n \bigcup_{k=1}^N \left( I_{i,k}^c \cup I_{i,k}^w \right)$. \par
%
%
%
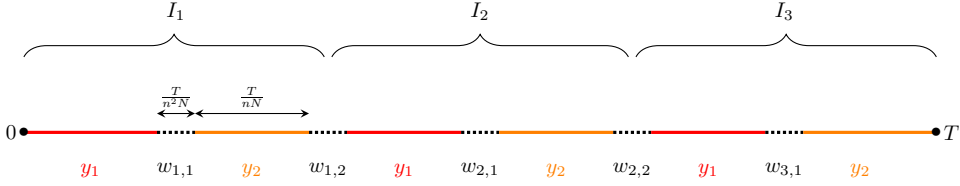
\begin{figure}[!ht]
\begin{center}
\begin{tikzpicture}[scale=0.5,every node/.style={scale=.8}]

\draw[color=red][very thick] (0,0) -- (3.5,0);
\draw[color=red] (1.75,-1) node{$y_1$};
\draw[densely dotted][very thick] (3.5,0) -- (4.5,0);
\draw (4,-1) node{$w_{1,1}$};
\draw[color=orange][very thick] (4.5,0) -- (7.5,0);
\draw[color=orange] (6,-1) node{$y_2$};

\draw[decoration={brace,amplitude=3mm}, decorate] (0,2) -- (7.9,2) coordinate[midway] (mid);
\draw (4,3.2) node{$I_1$};
\draw[decoration={brace,amplitude=3mm}, decorate] (8.1,2) -- (15.9,2) coordinate[midway] (mid);
\draw (12,3.2) node{$I_2$};
\draw[decoration={brace,amplitude=3mm}, decorate] (16.1,2) -- (24,2) coordinate[midway] (mid);
\draw (20,3.2) node{$I_3$};

\draw[<->,>=stealth] (3.5,0.5) -- (4.5,0.5);
\draw[<->,>=stealth] (4.5,0.5) -- (7.5,0.5);
\draw (4,0.5) node[above,scale=0.8]{$\frac{T}{n^2N}$};
\draw (6,0.5) node[above,scale=0.8]{$\frac{T}{nN}$};

\draw[densely dotted][very thick] (7.5,0) -- (8.5,0);
\draw (8,-1) node{$w_{1,2}$};
\draw[color=red][very thick] (8.5,0) -- (11.5,0);
\draw[color=red] (10,-1) node{$y_1$};
\draw[densely dotted][very thick] (11.5,0) -- (12.5,0);
\draw (12,-1) node{$w_{2,1}$};
\draw[color=orange][very thick] (12.5,0) -- (15.5,0);
\draw[color=orange] (14,-1) node{$y_2$};

\draw[densely dotted][very thick] (15.5,0) -- (16.5,0);
\draw (16,-1) node{$w_{2,2}$};
\draw[color=red][very thick] (16.5,0) -- (19.5,0);
\draw[color=red] (18,-1) node{$y_1$};
\draw[densely dotted][very thick] (19.5,0) -- (20.5,0);
\draw (20,-1) node{$w_{3,1}$};
\draw[color=orange][very thick] (20.5,0) -- (24,0);
\draw[color=orange] (22,-1) node{$y_2$};

\draw (0,0) node[left]{$0$};
\draw (0,0) node{$\bullet$}; 
\draw (24,0) node[right]{$T$};
\draw (24,0) node{$\bullet$};

\end{tikzpicture}
\end{center}
\caption{Time intervals and associated control value for $\protect N=2$ and $ \protect n=3$.}
\label{f.intervals}
\end{figure}
%
%
%
For each $i$ and $k$, we introduce a smooth transition curve $w_{i,k}^n : \ I_{i,k}^w \longrightarrow \ \R^2$ that connects $y_{k}(t_{i-1,k}-\frac{\delta T}{2n})$ to $y_{k+1}(t_{i-1,k}+\frac{\delta T}{2n})$ in a time $\frac{T}{n^2N}$. The precise construction of this transition curve is given in the next paragraph. 
When these trajectories are determined we define, for $n$ in $\N^*$, the continuous trajectory $z_n^0$:
\begin{equation} \label{e.control0}
z_n^0(t) := \sum_{i=1}^n \sum_{k=1}^N \left(\mathds{1}_{I_{i,k}^c}(t) y_k(t) + \mathds{1}_{I_{i,k}^w}(t) w_{i,k}^n(t) \right) \text{ for } t \in [0,T].
\end{equation}
The control $z_n$ will finally be defined as a regularization of $z_n^0$. \par
\ \par
\noindent
\textbf{3. Precise form of the transition curves.} We fix $1\leq i \leq n$ and $1 \leq k \leq N$; let us detail the construction of the Lipschitz transition curves $w_{i,k}^n \in C(I_{i,k}^w)$. To lighten the notation, we will temporarily write $t_1 := t_{i-1,k}-\frac{\delta T}{2n}$ and $t_2:= t_{i-1,k}+\frac{\delta T}{2n}$. \par
To connect $y_{k}(t_1)$ to $y_{k+1}(t_2)$, we consider the set of balls of center $\overline{x}_j(t_1)$ and of radius $3\overline{r}/2$ for $j\in \{1,\ldots,N\}$, which are all disjoint according to \eqref{e.ymin}. 
Moreover $y_k(t_1)$ does not belong to any of these balls according to \eqref{e.ymin}. 
We construct a trajectory $w_{i,k}^n$ avoiding the balls $B_{\overline{x}_j(t_1)}(3\overline{r}/2)$, $j \in \{1,\ldots,N\}$. 
To this end, we first define a Lipschitz trajectory $w_{i,k}^{n,\text{pm}}$ joining $y_{k}(t_1)$ to $y_{k+1}(t_2)$ along	 a straight line, and circumventing the ball $B_{\overline{x}_j(t_1)}(3\overline{r}/2)$ by following its boundary if the line crosses it, see Figure~\ref{f.wtrajectory}. We choose an arbitrary sense to follow the circle: for instance, we choose the shortest arc, and the clockwise direction when both arcs have the same length. The resulting curve is followed at constant speed. \par
%
%
\begin{figure}[!h]
\begin{center}
\begin{tikzpicture}[scale=0.6,every node/.style={scale=.8}]
\draw (2,6) -- (14,2);
\draw (2,6) node{$\bullet$};
\draw (14,2) node{$\bullet$};
\draw (2,6) node[above]{$y_{k}(t_1)$};
\draw (14,2) node[below]{$y_{k+1}(t_2)$};
\draw[white, fill=gray!30] (5,5) circle(1.5);
\draw[white, fill=gray!30] (11,4) circle(1.5);
\draw[gray] (5,3.5) node[below]{$B_{\overline{x}_i(t_1)}(\frac{3\overline{r}}{2})$};
\draw[gray] (5,5) node{$\bullet$};
\draw[gray] (5,5) node[above]{$\overline{x}_i(t_1)$};
\draw[gray] (11,4) node[above]{$\overline{x}_j(t_1)$};
\draw[gray] (11,4) node{$\bullet$};
\draw[gray] (11,2.5) node[below]{$B_{\overline{x}_j(t_1)}(\frac{3\overline{r}}{2})$};
\draw (6.43,4.54) arc (-18:162:1.5);
\draw (9.59,3.48) arc(-160:-57:1.5);
\draw (8,4) node[above]{$w_{i,k}^{n,\text{pm}}$};
\end{tikzpicture}
\end{center}
\caption{Construction of the trajectory $\protect w_{i,k}^{n,\text{pm}}$.}
\label{f.wtrajectory}
\end{figure}
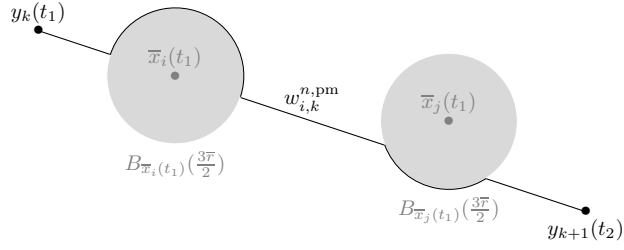
Let us remark that the trajectory $z_n^0$ defined by \eqref{e.control0} is Lipschitz on $[0,T]$; we call $L(n)$ its best Lipschitz constant. One can readily check that this $L(n)$ may be estimated, for some geometrical constant $C^{\pi}$, by
\begin{equation} \label{e.Lipschitzn}
L(n) \leq \frac{C^{\pi} Nn^2}{T} \sup_{k \in \{1,\dots,N\}} \|y_k\|_{C^1([0,T])}.
\end{equation} 
\par
\ \par
\noindent
\textbf{4. Regularization of the trajectory $z_n^0$.} Now we introduce the regularized control sequence $z_n$ from the Lipschitz construction $z_n^0$ above. Let $(\rho_{\varepsilon})_{\varepsilon>0}$ a mollifying sequence obtained from rescaling $\rho \in C_{c}^{\infty}((0,T),\R)$, $\rho \geq 0$, of integral $1$, and of support in $(-\varepsilon, \varepsilon)$ for $\varepsilon$ sufficiently small. For $n$ in $\N^*$, we define the function $z_{n}$ in  $C^{\infty}([0,T],\R^2)$ by
\begin{equation} \label{e.control}
z_n (t) := \rho_{\frac{1}{n^3}} * z_n^0 (t) \ \text{ in } [0,T],
\end{equation}
where we extended $z_{n}^0$ by  $z_{n}^0(0)$ (respectively  $z_{n}^0(T)$) on $(-\infty,0)$ (resp. $(T+\infty)$). 
Notice that since $\frac{1}{n^3} \ll \frac{1}{n}$ and since each $y_{k}$ is constant in $I_{1,1}^c$ and in $I_{n,N}^c$, we have $z_n (0)=z_n^0(0)$ and  $z_n (T)=z_n^0(T)$ for $n$ sufficiently large. \par
\ \par
\noindent
\textbf{5. A property of the control $z_n$.} We now state a property of the control $z_n$ following from its construction.
\begin{Lemm} \label{t.wconnection}
There exists $n_0 \in \N^*$ such that for all $n \geq n_0$, for all $t \in [0,T]$, and for all $j \in \{1,\ldots,N\}$,
\begin{equation} \label{e.wconnection}
|z_n(t) - \overline{x}_j(t) | \geq \overline{r},
\end{equation}
where $\bm{\overline{x}}$ is the trajectory solution of \eqref{e.ODExav} and  $\overline{r} > 0$ was introduced in \eqref{e.ymin}. 
\end{Lemm}
\begin{proof}[Proof of Lemma~\ref{t.wconnection}]
Using  $\mbox{Supp} \, \rho_{\frac{1}{n^3}} \subset(-\frac{1}{n^3},\frac{1}{n^3})$, we deduce for $t$ in $[0,T]$
\begin{equation*}
|z_n(t) - z_n^0(t)| \leq  \int_{\R} \rho_{\frac{1}{n^3}}(s) \left| z_n^0(t-s)- z_n^0(t) \right| \dd s \leq \frac{2L(n)}{n^3},
\end{equation*}
and thus with \eqref{e.Lipschitzn} the uniform estimate

\begin{equation} \label{e.speed}
\sup_{t\in [0,T]} |z_n(t) - z_n^0(t)| \leq \frac{2C^{\pi}N}{nT} \sup_{k \in \{1,\dots,N\}} \|y_k\|_{C^1([0,T])}.
\end{equation}

Now for $t$ in $I_{i,k}^c$, for all $i$ in $\{1,\dots,n\}$ and $k$ in $\{1,\dots,N\}$, according to \eqref{e.ymin} and to the construction of $z_n^0$, we have $|z_n^0(t) - \overline{x}_j(t)| \geq 2 \overline{r}$; consequently with \eqref{e.speed}, we see that for $n$ sufficiently large, \eqref{e.wconnection} holds on these intervals.\par
Concerning intervals $I_{i,k}^w$, we use the uniform continuity of $\bm{\overline{x}}$ on $[0,T]$ to deduce the existence of $n_u$ such that for all $n \geq n_{u}$, for all $t, t^{\prime} \in [0,T]$ with $|t-t^{\prime}| \leq \frac{\delta T}{n} = \frac{T}{n^2 N}$, for all $j \in \{1,\ldots,N\}$, 
\begin{equation*}
|\overline{x}_j(t) - \overline{x}_j(t^{\prime}) | \leq \overline{r}/3.
\end{equation*}
Hence for $n \geq n_u$, for all $t \in I_{i,k}^w$, for all $j \in \{1,\ldots,N\}$, $|\overline{x}_j(t) - \overline{x}_j(t_{i-1,k}-\frac{\delta T}{2n}) | \leq \overline{r}/3$.  As by construction $|z_n^0(t) - \overline{x}_j(t_1) | \geq 3\overline{r}/2$, we have that for all $t \in I_{i,k}^w$, 
\begin{equation} \label{e.wmin}
|z_n^0(t) - \overline{x}_j(t) | \geq\frac{7\overline{r}}{6}.
\end{equation}
From \eqref{e.speed} we deduce the existence of $n_w$ such that for $n$ sufficiently large, \eqref{e.wconnection} holds as well in intervals  $I_{i,k}^c$. This ends the proof of Lemma~\ref{t.wconnection}.
\end{proof}
\subsection{Approximate controllability}
\label{Subsec:AC}
\subsubsection{Statements}
We now focus on the trajectory $\bm{x^n}$ solution of the vortex system controlled by $z_n$, the control introduced in \eqref{e.control}. For $n$ in $\N^*$ and $i$ in $\{1,\ldots,N\}$, we denote $x_i^n$ the solution of
\begin{equation}\label{e.ODEzn}
\left\{
\begin{array}{cl}
\displaystyle \frac{dx_i^n}{dt}(t) &= \displaystyle \sum_{j \neq i} \gamma_j K* \delta_{x_j^n(t)}(x_i^n(t)) + \gamma^c K* \delta_{z_n(t)}(x_i^n(t)), \\
x_i^n(0) & = x_{0,i}.
\end{array}
\right.
\end{equation}
The goal of this subsection is to establish the following result concerning the approximate controllability of the system.
In the following we let $\bar{r}$ denotes the constant introduced in \eqref{e.ymin}.
\begin{Theo} \label{t.approcontrol}
There exists $n'_{0} \in \N^*$ and $C>0$ depending only on $\overline{r}$ and $T$, 
such that for any $n \geq n'_{0}$, the solution ${\bm x}^n$ is defined globally in $[0,T]$ and satisfies: 
\begin{equation} \label{Eq:xxbarre}
\| {\bm x}^n -\overline{\bm x}\|_{C^0([0,T])} \leq \frac{C}{n},
\end{equation}
where $\bm{\bar{x}}$ is the reference solution of~\eqref{e.ODExav}.
Hence System~\eqref{e.ode1control} is approximately controllable in arbitrary time.
\end{Theo}
The proof of Theorem~\ref{t.approcontrol} uses a continuous induction argument relying on a local result stated in Proposition~\ref{t.xnconvergence} below. For this argument, we introduce a hypothesis depending on some time $\tau>0$ and some integer $n_{1}$ as follows. 
\begin{Hypo}[H\ref{h.hyp}($\tau$,$n_{1}$)] \label{h.hyp}
For all $n \geq n_1$, the trajectory $\bm{x^n}$ solution of \eqref{e.ODEzn} is well defined on $[0,\tau]$ and  for all $t$ in $[0,\tau]$,
\begin{equation*}
\forall i \in \{1,\ldots,N\}, \  |z_n(t) - x_i^n(t) |\geq \overline{r}/2 \ \text{ and } \
\forall i,j \in \{1,\ldots,N\}, \  i\neq j, \  |x_j^n(t) - x_i^n(t) |\geq \overline{r}/2.
\end{equation*}
\end{Hypo}
We now state a convergence result valid when this hypothesis holds. 
We let $n_0$ denote the integer introduced in Lemma~\ref{t.wconnection}. 
\begin{Prop} \label{t.xnconvergence}
If Hypothesis H\ref{h.hyp}($\tau$,$n_{1}$) holds for given $\tau>0$ and $n_{1} \geq n_{0}$, then there exists $\tilde{C}>0$ depending only on $\bar{r}$ such that 
for all $n\geq n_1$, for all $i$ in $\{1,\ldots,N\}$,
\begin{equation} \label{e.xnconvergence}
\max_{t \in [0,\tau]} |x^n_i(t) - \overline{x}_i(t)| \leq \frac{\tilde{C}}{n},
\end{equation}
where $\bm{\overline{x}}$ is the reference solution of \eqref{e.ODExav}.
\end{Prop}
The proof of Proposition~\ref{t.xnconvergence} is the goal of next paragraph. Then we will be able to prove Theorem~\ref{t.approcontrol}. \par

\subsubsection{Proof of Proposition~\ref{t.xnconvergence}}
The proof of Proposition~\ref{t.xnconvergence} relies on a preliminary result which we will state after introducing a few notations.
In the following, $M^1(\R^2)$ denotes the set of bounded measures of $\R^2$ as the dual space of $C_0(\R^2)$ the space of continuous functions that vanish at infinity endowed with the uniform norm.
The set $C^1_0(\R^2)$ denotes the subspace of $C^1 \cap C_{0}$ functions whose derivative belongs to $C_0(\R^2)$, endowed with the norm $\|f\|_{C^1_0(\R^2)} = \sup_{x \in \R^2} |f(x)| + \sup_{x \in \R^2}|\nabla f(x) |$.
In addition we consider the space $BV(0,T;C_0(\R^2))$ of functions of bounded variations in time, with values in $C_0(\R^2)$. For $\phi \in BV(0,T;C_0(\R^2))$, we will write:
\begin{gather*}
|\phi|_{BV(0,T;C_0(\R^2))}  := \sup_{M \geq 1} \ \sup_{0 \leq a_1 < \ldots < a_{M+1} \leq T} \ \sum_{i=1}^M \|\phi(a_{i+1}) - \phi(a_i)\|_{C_0(\R^2)}, \\
\|\phi\|_{BV(0,T;C_0(\R^2))}  := |\phi|_{BV(0,T;C_0(\R^2))} + \|\phi\|_{L^\infty(0,T;C_0(\R^2))} .
\end{gather*}
In the space $\text{L}^{\infty}(0,T;M^1(\R^2))$ of bounded time-dependent Radon measures, we consider the sequence  $(\delta_{z_n(\cdot)})_{n\in\N}$ and introduce 
\begin{equation*}
\overline{\delta}_{(\cdot)} := \frac{1}{N}\sum_{k=1}^N \delta_{y_k(\cdot)}.
\end{equation*}
Then we have the following quantitative weak convergence result.

\begin{Lemm} \label{t.weakconv}
Consider the sequence $(z_n)_{n \in \N^*}$ of $C^{\infty}([0,T],\R^2)^{\N^*}$ defined by \eqref{e.control}. For $\phi$ in $BV(0,T;C_0(\R^2))\cap \text{L}^1(0,T;C_0^1(\R^2))$, there exists $C>0$ depending on $\|y_k\|_{C^1(0,T)}$ for $1\leq k \leq N$, such that for all $n \in \N^*$,
\begin{equation} \label{e.weakconv}
\int_0^T \left| \left\langle \delta_{z_n(t)},\phi(t,\cdot) \right\rangle - \left\langle \overline{\delta}_{(t)},\phi(t,\cdot) \right\rangle  \right| \dd t \leq \frac{C}{n} \left(\|\phi\|_{BV(0,T;C_0(\R^2))} + \|\phi \|_{\text{L}^1(0,T,C^1_0(\R^2))} \right).
\end{equation}
with $\langle.,.\rangle$ the duality bracket between $M^1(\R^2)$ and $C_0(\R^2)$.
\end{Lemm}
The proof of Lemma~\ref{t.weakconv} is postponed to Subsubection~\ref{sub.weakconv}. 
Let us now establish Proposition~\ref{t.xnconvergence}.
\begin{proof}[Proof of Proposition~\ref{t.xnconvergence}]
Let us introduce notations for the vector fields appearing in the right-hand sides of \eqref{e.ODEzn} and \eqref{e.ODExav}: 
\begin{equation}\label{e.Fndef}
F_{\bm{x^n},i}(t,\cdot):=  \sum_{j \neq i} \gamma_j K* \delta_{x_j^n(t)} + \gamma^c K* \delta_{z_n(t)} \text{ and } 
\overline{F}_{\bm{\bar{x}},i}(t,\cdot) :=  \sum_{j \neq i} \gamma_j K* \delta_{\overline{x}_j(t)} + \frac{\gamma^c}{N} \sum_{k=1}^N K* \delta_{y_k(t)}.
\end{equation}
As for instance in \cite[Section 4.2]{MP}, we consider, for $\eta < 1$, a radial function $\text{ln}_{\eta}$ of class $C^{\infty}(\R^2)$, satisfying the three following properties:
\begin{equation*}
\text{ln}_{\eta}(x) = \text{ln}(|x|) \; \text{for } |x|>\eta, \
|\text{ln}_{\eta} (x)|  \leq |\text{ln}(|x|) | \; \text{for } x\in \R^2 \ \text{ and } \
\left| \nabla \text{ln}_{\eta} (x) \right| \leq |x|^{-1}\; \text{for } x\in \R^2. 	
\end{equation*}
Then we define the regularized Biot-Savart Kernel $K_{\eta} = \nabla^{\perp} \text{ln}_{\eta} $ of $C^{\infty}(\R^2)$, satisfying $K_{\eta}(x) = K(x)$ for $|x| > \eta$, and globally Lipschitz on $\R^2$. 
This allows to define regularized vector fields $F^\eta_{n,i}$ and $\overline{F}^\eta_i$ by
\begin{equation*}
F^\eta_{\bm{x^n},i}(t,\cdot) := \sum_{j \neq i} \gamma_j K_{\eta}* \delta_{x^n_j(t)} + \gamma^c K_{\eta}* \delta_{z_n(t)}
\text{ and } 
\overline{F}^\eta_{\bm{\bar{x}},i}(t,\cdot) :=  \sum_{j \neq i} \gamma_j K_{\eta}* \delta_{\overline{x}_j(t)} + \frac{\gamma^c}{N} \sum_{k=1}^N K_{\eta}* \delta_{y_k(t)}.
\end{equation*}
Note that $F^\eta_{\bm{x^n},i}$ and $\overline{F}^\eta_{\bm{\bar{x}},i}$ are continuous in time and globally Lipschitz in space and hence their flows are global.
Assuming that Hypothesis H\ref{h.hyp}($\tau$,$n_1$) holds, then for
\begin{equation*} 
\eta: = \min(\overline{r}/4,1/2),
\end{equation*}
and $n \geq n_1$, the solution associated with $F^\eta_{\bm{x^n},i}$ coincides with the solution of \eqref{e.ODEzn}. Then we can replace $K$ with $K_{\eta}$ in \eqref{e.ODExav}.  \par
Now for $j \in\{1,\ldots,N\}$ and $t$ in $[0,\tau]$, we let $\phi^j_{t}$ the following vector-valued function
\begin{equation*} 
\phi^j_{t}(s,y) := \mathds{1}_{[0,t[}(s) \, K_{\eta}(\overline{x}_{j}(s)-y).
\end{equation*}
Clearly, $\phi^j _{t}$ belongs to $BV(0,T;C_0(\R^2)) \cap L^1(0,T;C_0^1(\R^2))$, is right-continuous and
the following relation holds:
\begin{equation*}
\int_0^{\tau} \left| \left\langle \delta_{z_n(s)}, \phi^j_{t}(s,\cdot) \right\rangle 
- \left\langle \overline{\delta}_{(s)},\phi^j_{t}(s,\cdot) \right\rangle \right| \dd s 
=  \int_0^{t} |K_{\eta}*(\delta_{z_n(s)} - \overline{\delta}_{(s)})|(\overline{x}_{j}(s))  \dd s.
\end{equation*}
Hence with Lemma~\ref{t.weakconv} we deduce the existence of $C>0$ depending on $\eta$ but independent of  $j$, $\tau$ and $t$ such that the following holds for all $n \geq n_1$:
\begin{equation} \label{e.ineponct}
\int_0^{t} |K *(\delta_{z_n(s)} - \overline{\delta}_{(s)})|(\overline{x}_{j}(s)) \dd s \leq \frac{C}{n}.
\end{equation}
Now we consider the distance between the trajectory $x_i^n$ and its reference trajectory $\bar{x}_i$, for $i$ in $\{1, \dots, N \}$,
\begin{equation*}
\begin{split}
|x^n_i - \overline{x}_i|(t) & \leq \left| \int_0^t \left(F_{\bm{x^n},i}(s,\overline{x}_{i}(s))-\overline{F}_{\bm{\bar{x}},i}(s,\overline{x}_{i}(s)) \right)\dd s \right| + \left|\int_0^t \left(F_{\bm{x^n},i}(s,x^n_{i}(s))-F_{\bm{x^n},i}(s,\overline{x}_{i}(s))\right)\dd s \right|.
\end{split}
\end{equation*}
For the first term in the right-hand side, recalling the relation $\left| \frac{1}{z}-\frac{1}{z^{\prime}} \right| = \frac{|z-z^{\prime}|}{|zz^{\prime}|}$ and \eqref{e.ymin}, we have
\begin{equation} \label{inversionterm}
\sum_{j \neq i} K *\left( \delta_{x_j^n(t)} - \delta_{\bar{x}_j(t)} \right) (\bar{x}_i(t)) \leq \sum_{j\neq i } \frac{|x_j^n- \bar{x}_j|}{|x_j^n-\bar{x}_i| |\bar{x}_i - \bar{x}_j|} \leq \frac{1}{(4\bar{r})^2}\sum_{j\neq i}  |x_j^n- \bar{x}_j|.
\end{equation}
Under Hypothesis H\ref{h.hyp}, $F_{\bm{x^n},i}^{\eta}$ is globally $L$-Lipschitz in space, with $L$ depending on $\eta$, hence on $\overline{r}$. Defining $\tilde{L} := L + \frac{1}{(4 \bar{r})^2}$ and with \eqref{inversionterm}, one has for $n \in \N^*$, $t \in [0,\tau]$ and $i \in \{1,\ldots,N\}$,
\begin{equation*}
|x_i^n-\bar{x}_i |(t) \leq |\gamma^c| \tilde{L} \sum_{j \neq i} \int_0^t |x_j^n-\bar{x}_j |(s) \dd s + |\gamma^c| \int_0^t  \left| K*(\delta_{z_n(s)}-\overline{\delta}_{(s)})(\overline{x}_j(s)) \right| \dd s.
\end{equation*}
According to Grönwall's lemma, this leads to
\begin{equation*}
|x^n_i - \overline{x}_i|(t)  \leq |\gamma^c| e^{\tilde{L} T} \sum_{j=1}^N \int_0^t |K*(\delta_{z_n(s)}-\overline{\delta}_{(s)}) \, (\overline{x}_{j}(s))| \dd s .
\end{equation*}
With \eqref{e.ineponct}, we conclude that for some $\tilde{C}$ independent of $\tau$, \eqref{e.xnconvergence} holds for all $n \geq n_1$. 
This ends the proof of Proposition~\ref{t.xnconvergence}.
\end{proof} 
\subsubsection{Proof of Theorem~\ref{t.approcontrol}}
\begin{proof}[Proof of Theorem~\ref{t.approcontrol}]
Establishing Theorem~\ref{t.approcontrol} amounts to showing that Hypothesis H\ref{h.hyp}($\tau$,$n_{1}$) holds for $\tau = T$ and $n_{1}$ large enough; as outlined before, we proceed by continuous induction. \par
For $n$ in $\N^*$, we define
\begin{multline*}
T_n :=\sup \left\{ \tau \in [0,T] \ \Big| \ \forall t \in[0,\tau], \ \forall i \in \{1,\ldots,N\},\ \forall i \neq j, \right. \\
 \left. |z_n(t) - x^n_i(t)|> \overline{r}/2 \ \text{ and } \ |x^n_j(t) - x^n_i(t)|> \overline{r}/2 \right\}.
\end{multline*}
Thanks to \eqref{e.ymin} and according to the continuity of $z_n$ and $\bm{x}^n$, each $T_{n}$ is positive for $n \in \N^*$. \par
\ \par
\noindent
{\bf 1.} The first step of this proof consists in showing the existence of $\underline{T}>0$ such that for all $n \geq n_{0}$ (where $n_{0}$ was introduced in Lemma~\ref{t.wconnection}), $T_{n} \geq	\underline{T}$.\par
Let $n \geq n_{0}$. It is straightforward that $\bm{x^n}$ is well-defined on $[0,T_n]$. Relying on \eqref{e.ODEzn}, we have that for $t \in [0,T_{n}]$, 
\begin{equation*}
\left| \frac{d x_i^n(t)}{dt} \right | \leq \frac{2}{\overline{r}} \left(|\gamma^c| +  \sum_{j=1}^N |\gamma_j| \right).  
\end{equation*}
Thus setting 
\begin{equation*}
\underline{T}_1 := \displaystyle \frac{\overline{r}^2}{\displaystyle 16\left(|\gamma^c| +  \sum_{j=1}^N |\gamma_j| \right)}
\end{equation*}
we find that for $t \in [0,\underline{T}_1] \cap [0,T_{n}]$ and $i \in \{1,\ldots,N\}$,
\begin{equation*}
|x^n_i(t) - x_i(0) | \leq \frac{\overline{r}}{8}.
\end{equation*}
Consequently, for $i \neq j$, recalling that $\overline{\bm x}(0)={\bm x}(0)$ so that \eqref{e.ymin} gives $|x_{i}(0) - x_{j}(0)| \geq 4\overline{r}$, we have
\begin{equation*}
|x^n_j(t) - x^n_i(t) | \geq \frac{15\overline{r}}{4}.
\end{equation*}
Concerning the reference solution $\overline{\bm x}$, by a continuity argument, there exists $\underline{T}_2 >0$ such that for all $t  \in [0, \underline{T}_2]$, for all $i \in \{1, \ldots, N\}$, $|\overline{x}_i(t) - \overline{x}_i(0) | \leq \overline{r}/8$. 
Using \eqref{e.wconnection} we deduce that on $[0,\underline{T}_2]$, for $n \geq n_0$, 
\begin{equation*}
|z_n(t)- x^n_i(t)| \geq |z_n(t)- \overline{x}_i(t)| - |x^n_{i}(t)- x^n_{i}(0)| - |\overline{x}_i(t) -\overline{x}_i(0)|
\geq 
\frac{3\overline{r}}{4}.
\end{equation*}
Hence by an immediate contradiction argument, we obtain that $T_{n} \geq \underline{T}:=\min(\underline{T}_1, \underline{T}_2)$. \par 
\ \par
\noindent
{\bf 2.} Now let us show that $T_{n}=T$ for $n$ sufficiently large. We define
\begin{equation*} 
\tilde{T}:= \inf_{n \geq n_{0}} T_{n}.
\end{equation*}
We know $\tilde{T}>0$; let us show that $\tilde{T}=T$. By definition of $T_{n}$, we see that the hypothesis H\ref{h.hyp}($\tilde{T}$,$n_{0}$) is fulfilled. In particular we deduce a constant $\tilde{C}$ according to Proposition~\ref{t.xnconvergence}. Now we define $n'_{0}$ as
\begin{equation*} 
n'_{0}:=\max\left( n_{0}, \left\lceil \frac{4\tilde{C}}{\overline{r}} \right\rceil  \right).
\end{equation*}
We claim that for all $n \geq n'_{0}$, $T_{n}=T$. Indeed, applying Proposition~\ref{t.xnconvergence} we find that for all $n \geq n'_{0}$,
\begin{equation*} 
\max_{t \in [0,\tilde{T}]} |x^n_i(t) - \overline{x}_i(t)| \leq \frac{\overline{r}}{4}.
\end{equation*}
With \eqref{e.ymin} and Lemma~\ref{t.wconnection}, this involves that for $t \in [0,\tilde{T}]$, $i \in \{1,\ldots,N\}$, $j \neq i$,
\begin{equation*} 
|x^n_i(t) - z_n(t)| \geq \frac{3\overline{r}}{4} \ \text{ and } \ |x^n_i(t) - x^n_j(t)| \geq \frac{3\overline{r}}{4}.
\end{equation*}
If we had $\tilde{T} < T$, we could find a $n$ contradicting one of these inequalities at time $t=T_{n}$. \par
It follows that for $n \geq n'_{0}$, the solution ${\bm x}^n$ is defined globally in $[0,T]$. Moreover H\ref{h.hyp}($T$,$n'_{0}$) is true. Hence \eqref{Eq:xxbarre} follows from Proposition~\ref{t.xnconvergence}. \par
This ends the proof of Theorem~\ref{t.approcontrol}.
\end{proof}
\subsubsection{Proof of Lemma~\ref{t.weakconv}}
\label{sub.weakconv}
Recalling the definition of $z_{n}^0$ in \eqref{e.control0}, we write
\begin{multline} \nonumber
\int_0^T \left| \langle \delta_{z_n(t)},\phi(t,\cdot) \rangle - \langle \overline{\delta}_{(t)},\phi(t,\cdot) \rangle \right| \dd t
\leq
\int_0^T | \left\langle \delta_{z_n(t)},\phi(t,\cdot) \right\rangle - \left\langle \delta_{z_n^0(t)},\phi(t,\cdot) \right\rangle | \dd t 
\\ +
 \int_0^T | \left\langle \delta_{z_n^0(t)},\phi(t,\cdot) \right\rangle - \left\langle \overline{\delta}_{(t)},\phi(t,\cdot) \right\rangle | \dd t,
\end{multline}	
and estimate the two terms in the right-hand side separately. \par
\ \par 
\noindent
$\bullet$ For the first term, we use \eqref{e.speed} and deduce that we have
\begin{align*}
 \int_0^T | \left\langle \delta_{z_n(t)},\phi(t,\cdot) \right\rangle - \left\langle \delta_{z_n^0(t)},\phi(t,\cdot) \right\rangle | \dd t & \leq \sup_{t\in [0,T]} |z_n(t) - z_n^0(t)| \int_0^T | \phi(t,\cdot)|_{C_0^1(\R^2)} \dd t  \\
& \leq \frac{2C^{\pi}N}{nT} \sup_{k \in \{1,\dots,N\}} \|y_k\|_{C^1([0,T])} \|\phi \|_{\text{L}^1(0,T,C^1_0(\R^2))} .
\end{align*}
\ \par 
\noindent
$\bullet$ The main part concerns the second term. We set for $j \in \{1,\ldots,N\}$, 
\begin{equation*}
\tilde{z}_n^j(t):= \sum_{i=1}^n \sum_{k=1}^N \mathds{1}_{I_{i,k}}(t) y_{k+j}(t),
\end{equation*}
with the convention that $k+j = k+j-N$ if $k+j > N$. 
We extend $\phi$ for all times by setting $\phi(t,\cdot) = \phi(0,\cdot)$ for $t\leq 0$  $\phi(t,\cdot) = \phi(T,\cdot)$ for $t\geq T$.
We claim the following.
\begin{Lemm} \label{Lem:zztilde}
For some constant $C>0$ depending on $y_{k}$, we have
\begin{equation*}
\int_0^T \left| \left\langle \delta_{z_n(t)},\phi(t,\cdot) \right\rangle \dd t - \frac{1}{N} \sum_{j=1}^N \left\langle \delta_{\tilde{z}_n^j(t)},\phi(t,\cdot) \right\rangle \right| \dd t \leq \frac{C}{n} \left(\|\phi\|_{BV(0,T;C_0(\R^2))}  + \|\phi \|_{\text{L}^1(0,T,C^1_0(\R^2))} \right).
\end{equation*}
\end{Lemm}
\begin{proof}[Proof of Lemma~\ref{Lem:zztilde}]
Using the change of time variable $s = t + j\delta T$ and an index permutation, we can write

\begin{align*}
\int_0^T \left\langle \delta_{\tilde{z}_n^j(t)},\phi(t,\cdot) \right\rangle \dd t 
= &\int_0^T \sum_{i=1}^n \sum_{k=1}^N \mathds{1}_{I_{i,k}}(t) \phi(t,y_{k+j}(t)) \dd t \\
= &\int_0^T \sum_{i=1}^n \sum_{k=1}^N \mathds{1}_{I_{i,k}}(s) \phi(s-j\delta T,y_{k}(s-j\delta T)) \dd s 
+ \mathcal{E},
\end{align*}
where the term $\mathcal{E}$ comes from the errors at the boundaries $t=0$ and $t=T$ where the time change $s = t + j\delta T$ does not fit, and can easily be bounded as follows:
\begin{equation*} 
|\mathcal{E}| \leq \frac{T \|\phi\|_{L^\infty(0,T;C_{0}(\R^2))}}{n}.
\end{equation*}
We infer
\begin{multline} \label{e.decompspeed}
\int_0^T \left| \left\langle \delta_{z_n^0(t)},\phi(t,\cdot) \right\rangle  - \left\langle \delta_{\tilde{z}_n^j(t)},\phi(t,\cdot) \right\rangle \right| \dd t \\
\leq  \displaystyle \int_0^T \sum_{i=1}^n \sum_{k=1}^N \mathds{1}_{I_{i,k}}(t) \left| \phi(t, y_k(t)) - \phi\big(t-j\delta T, y_k(t-j\delta T)\big) \right|  \dd t \hspace{1.5cm} \\
+ \int_0^T \sum_{i=1}^n \sum_{k=1}^N \left( \mathds{1}_{(I_{i,k}\setminus I_{i,k}^c)} |\phi(t, y_k(t))|  + \mathds{1}_{I_{i,k}^w}(t) |\phi(t, w_k(t))|\right)  \dd t + \mathcal{E}.
\end{multline}
Let us now estimate of the first two terms in the right-hand side of \eqref{e.decompspeed}. \par
\ \par
\noindent
{\bf 1.} Concerning the first term, let us define for $t \in (0,T)$:
\begin{equation*}
g_n(t) := \displaystyle\sum_{i=1}^n \sum_{k=1}^N \mathds{1}_{I_{i,k}}(t) \left| \phi(t, y_k(t))- \phi(t-j\delta T, y_k(t)) \right|,
\end{equation*}
and 
\begin{equation*}
h_n(t) := \displaystyle\sum_{i=1}^n \sum_{k=1}^N \mathds{1}_{I_{i,k}}(t) \left| \phi(t-j\delta T, y_k(t))- \phi(t-j\delta T, y_k(t-j\delta T)) \right|.
\end{equation*}
As the intervals $I_{i,k}$ are disjoint, for $t$ in $(0,T)$ there exists some $k_t \in \{1,\ldots,N\}$ such that 
\begin{equation*}
g_n(t) = |\phi(t-j\delta T, y_{k_t}(t)) - \phi(t, y_{k_t}(t))| \leq  \|\phi(t-j\delta T, \cdot) - \phi(t, \cdot)\|_{C_{0}(\R^2)}.
\end{equation*}
We use the general property of functions of bounded variations (see e.g. \cite[Lemma 2.3]{Bressan:Livre}): for $u \in BV(\R;\R)$ and $\tau>0$,
\begin{equation*} 
\frac{1}{\tau} \int_{-\infty}^{+\infty} |u(x+\tau) - u(x)| \leq TV(u).
\end{equation*}
Hence  we deduce, recalling \eqref{deltaT},
\begin{equation*}
\int_0^T g_n(t) \dd t \leq j\delta T \ |\phi|_{BV(0,T;C_0(\R^2))}\leq \frac{T}{n}|\phi|_{BV(0,T;C_0(\R^2))}.
\end{equation*}
\par
\ \\
Moreover as $\phi(t, \cdot) \in C^1(\R^2)$ for almost all $t$ in $[0,T]$ and $y_k\in C^1([0,T])$ for $1\leq k\leq N$, we find
\begin{equation*}
h_n(t) \leq \frac{T}{n} \left(\sum_{k=1}^N \|y_k\|_{C^1([0,T])} \right) |\phi(t-j\delta T,.)|_{C^1_0(\R^2)},  
\end{equation*}
hence
\begin{equation*}
\int_0^T h_n(t) \dd t \leq \frac{T}{n} \left(\sum_{k=1}^N \|y_k\|_{C^1([0,T])} \right) \|\phi \|_{\text{L}^1(0,T;C^1_0(\R^2))}.
\end{equation*}
\ \par
\noindent
{\bf 2.} For what concerns the second term of \eqref{e.decompspeed}, we simply write 
\begin{multline*}
\int_0^T \sum_{i=1}^n \sum_{k=1}^N \left( \mathds{1}_{(I_{i,k}\setminus I_{i,k}^c)} |\phi(t, y_k(t))|  + \mathds{1}_{I_{i,k}^w}(t) |\phi(t, w_k(t))|\right)  \dd t \\
\leq  \|\phi\|_{\text{L}^{\infty}(0,T;C_0(\R^2))} \sum_{i=1}^n \sum_{k=1}^N \left( \lambda(I_{i,k}\setminus I_{i,k}^c)+ \lambda(I_{i,k}^w) \right)
\leq  \frac{2T}{n} \| \phi \|_{\text{L}^{\infty}(0,T;C_0(\R^2))},
\end{multline*}
with $\lambda(I)$ denoting the Lebesgue measure of the set $I$. \par
\ \par
\noindent
Gathering the inequalities above we obtain
\begin{equation*}
\displaystyle \int_0^T \left| \left\langle \delta_{z_n(t)},\phi(t,\cdot) \right\rangle -  \left\langle \delta_{\tilde{z}_n^j(t)},\phi(t,\cdot) \right\rangle \right| \dd t \leq \frac{C}{n},
\end{equation*}
for
\begin{multline*}
C:=  \displaystyle T\left( |\phi|_{BV(0,T;C_0(\R^2))} +  \| \phi \|_{\text{L}^1(0,T;C_0^1(\R^2))}\sum_{k=1}^N \|y_k\|_{C^1([0,T])} + 3 \| \phi \|_{\text{L}^{\infty}(0,T;C_0(\R^2))} \right) \\
 + \frac{2C^{\pi} N}{T} \sup_{k \in \{1,\dots,N\}} \|y_k\|_{C^1_0(\R^2)}  \|\phi \|_{\text{L}^1(0,T,C^1_0(\R^2))}.
\end{multline*}

Since this estimate holds for all $j$ in$\{1,\ldots,N\}$, this allows to establish Lemma~\ref{Lem:zztilde}.
\end{proof}
\noindent
To conclude the proof of Lemma~\ref{t.weakconv}, it remains to notice that 
\begin{equation*}
\frac{1}{N} \sum_{j=1}^N \int_0^T \left\langle \delta_{\tilde{z}_n^j(t)},\phi(t,\cdot) \right\rangle \dd t = \frac{1}{N} \sum_{j=1}^N \int_0^T \left\langle \delta_{y_j(t)},\phi(t,\cdot) \right\rangle \dd t 
= \int_0^T \left\langle \overline{\delta}_{(t)},\phi(t,\cdot) \right\rangle \dd t,
\end{equation*}
so \eqref{e.weakconv} follows. This ends the proof of Lemma~\ref{t.weakconv}.

\subsection{Exact controllability}
The goal of this subsection is to prove Theorem~\ref{Thm:Main}.

We will use a topological argument as in \cite{GlassRosier:2013} to pass from the approximate controllability to the exact one. This relies on the following lemma, see \cite[Lemma 4.1]{GlassRosier:2013}.
\begin{Lemm} \label{t.topo}
Let $\bm{w_0} \in (\R^2)^N$, $\kappa >0$, $f: \overline{B}(\bm{w_0},\kappa) \longrightarrow (\R^2)^N$ a continuous map such that we have $||f(\bm{w})-\bm{w}\| \leq \kappa/2$ for any $\bm{w}$ in $\partial B(\bm{w_0},\kappa)$. Then $B(\bm{w_0},\kappa/2) \subset f(\overline{B}(\bm{w_0},\kappa))$.
\end{Lemm}
Now we detail the proof of Theorem~\ref{Thm:Main}, relying on Lemma~\ref{t.topo}.
\begin{proof}[Proof of Theorem~\ref{Thm:Main}] The proof is divided in several steps. \par
\ \par 
\noindent
\textbf{1. Reduction to particular settings.}
Due to the approximate controllability result Theorem~\ref{t.approcontrol}, we can reduce the global exact controllability problem to a local one, namely when the initial and final positions of the vortices are in the situation described in Theorem~\ref{Var:Nvortex}. To be more precise, with $\bm{x_{f}}$ given, we introduce $D>0$, $\tau>0$ (arbitrarily small) and $\rho>0$ so that the result of Theorem~\ref{Var:Nvortex} applies. Now we use the approximate controllability Theorem~\ref{t.approcontrol} during a first step (of duration $T-\tau$) to bring the vortices close to their target, in such a way in particular that at time $T-\tau$, each vortex $x_{i}$ is in $\overline{B}(x_{f,i},D) \setminus \overline{B}(x_{f,i},D/2)$. \par
Once we are in this situation, we can begin a second phase (of duration $\tau$), where we again call $x_{0,1},\ldots,x_{0,N}$ the initial vortex positions, now in $\big(\overline{B}(x_{f,1},D) \setminus \overline{B}(x_{f,1},D/2)\big) \times \cdots \times \big(\overline{B}(x_{f,i},D) \setminus \overline{B}(x_{f,N},D/2)\big)$.
According to Theorem~\ref{Var:Nvortex}, we obtain the $N$ reference control trajectories $y_i$, $i \in \{1,\dots, N\}$ and the corresponding straight-lined reference solutions $\overline{x}_{i}$,  $i \in \{1,\dots, N\}$ of \eqref{e.odeN}, which belong to ${\mathcal Y}_{i}$ and ${\mathcal X}_{i}$, respectively.
Moreover we can extend the construction to the case where we replace ${\bm x}_{f}$ by any other point $\tilde{\bm x}_{f}$ in $B(x_{f,1},\rho) \times \cdots \times B(x_{f,N},\rho)$. \par
\ \par
\ \par 
\noindent
\textbf{2. Specific form of the transition curves.} 
Now we mainly follow the lines of the construction of Section~\ref{s.defcontrol} but modify the transition curves $w_{i,k}^n$ appearing in the construction of the oscillating control in \eqref{e.control0}, in order to make the construction continuous with respect to ${\bm x}_{f}$; this is easier in the present situation since we do not have to avoid the moving balls as in Figure~\ref{f.wtrajectory}. \par
To begin with, we choose some reference points $y_{i}^\star \in {\mathcal Y}_{i}$, $i \in \{1,\ldots,N\}$.
As the sets ${\mathcal X}_{i}$, ${\mathcal Y}_{j}$ are all disjoint for $i,j \in \{1,\dots, N\}$, according to an argument of path-connectedness of the plane deprived of disjoint convex compact sets, there exists a set of non crossing paths $\mathcal{C}_{k,k+1}$, $k \in \{1,\ldots,N\}$, such that $\mathcal{C}_{k,k+1}$ connects $y_{k}^\star$ to $y_{k+1}^\star$ (where $y_{N+1}^\star:=y_{1}^\star$) and avoids all other sets ${\mathcal X}_{i}$ and ${\mathcal Y}_{j}$. 
Then one constructs a Lipschitz transition curve $w_{i,k}^n$ the following way: straight from $y_{k}(t_{i-1,k}-\frac{\delta T}{2n})$ to $y_k^\star$, then following $\mathcal{C}_{k,k+1}$ between $y_k^\star$ and $y_{k+1}^\star$, and finally straight from $y_{k+1}^\star$ to $y_{k+1}(t_{i-1,k}+\frac{\delta T}{2n})$; this will be regularized as in Section~\ref{s.defcontrol}. \par
A particular feature of this construction is that the oscillating control $z_n$ defined by \eqref{e.control} now depends continuously on the objective final point $\tilde{\bm x}_f$ for all $\tilde{\bm x}_f \in \overline{B}({x}_{f,1},\rho) \times \cdots \times \overline{B}({x}_{f,N},\rho)$. \par 
\ \par 
\noindent
\textbf{3. Application of Lemma~\ref{t.topo}.}
Now we use the lines $\overline{x}_1,\dots, \overline{x}_N$ as reference trajectories; we find a constant $\overline{r}>0$ as in \eqref{e.ymin}, uniformly valid for any final point $\tilde{\bm x}_f \in \overline{B}({x}_{f,1},\rho) \times \cdots \times \overline{B}({x}_{f,N},\rho)$. From this constant $\overline{r}$, we deduce a constant $C>0$ and a rank $n'_{0} \in \N^*$ as in Theorem~\ref{t.approcontrol}. \par
Now we let
\begin{equation} \label{Constantes}
\kappa := \rho \ \text{ and } \ n_{\kappa} := \max \left( \left\lceil \frac{2 C}{\kappa}  \right\rceil, n'_0 \right)
\end{equation}
We define the following mapping 
\begin{equation}
f: 
\left\{ \begin{array}{r l}
  \tilde{\bm x}_f \in \overline{B}({x}_{f,1},\rho) \times \cdots \times \overline{B}({x}_{f,N},\rho)  \longrightarrow  & (\R^2)^N \\
   \tilde{\bm{x}}_f  \longmapsto &  \bm{x}^{n_{\kappa}}(T),
\end{array} \right.
\end{equation}
with $\bm{x}^{n_{\kappa}}(T)$ the final point of the trajectory controlled by $z_{n_{\kappa}}$ starting from ${{\bm x}}_{0}$. \par
Due to \eqref{Eq:xxbarre} and \eqref{Constantes}, we have $\|f(\tilde{\bm{x}}_f) - \tilde{\bm{x}}_f \| \leq \kappa/2$ for all $\tilde{\bm{x}}_f \in \overline{B}({x}_{f,1},\rho) \times \cdots \times \overline{B}({x}_{f,N},\rho)$. Moreover $f$ is continuous, as $z_{n_{\kappa}}$ depends continuously on the trajectories $y_k$ for $k \in \{1,\ldots,N\}$ according to the specific construction detailed above, which are continuously constructed from $\bm{x_f}$ since they are merely straight lines. \par
Hence with Lemma~\ref{t.topo}, we deduce that all points in $B(\bm{x}_f,\kappa/2)$ have a pre-image by $f$. Thus targeting $\bm{\tilde{x}_f}= f^{-1}(\bm{x_f})$, the trajectory $\bm{x}$ solution of \eqref{e.ode1control} reaches $\bm{x_f}$ in time $T$. \par
This ends the proof of Theorem~\ref{Thm:Main}.
\end{proof}
%
%
%
%
%
%
%
%
%
%
%
%

\end{document}